\documentclass [reqno]{amsart}

\usepackage{amsfonts}
\usepackage{amssymb}
\usepackage{color}
\usepackage{hyperref}
\newtheorem{thm}{Theorem}[section]
\newtheorem{cor}[thm]{Corollary}
\newtheorem{lem}[thm]{Lemma}

\theoremstyle{definition}
\theoremstyle{remark}
\newtheorem{rem}[thm]{Remark} %[section]

\numberwithin{equation}{section}

\newcommand{\C}{\mathbb{C}}
\newcommand{\comp}{\mathrm{comp}}

\newcommand{\iD}{\mathit{\Delta}}

\newcommand{\Id}{I}

\newcommand{\iPsi}{\mathit{\Psi}}

\begin{document}

\title[On non-coercive mixed problems]
      {On non-coercive mixed problems for parameter-dependent elliptic operators\footnote{This is a preprin version of the paper published in Mathematical Communications, 20 (2015), 131-150}}

%% First author

\author{A. Polkovnikov}

\address[Alexander Polkovnikov]
        {Siberian Federal University
\\
         Institute of Mathematics and Computer Science
\\
         pr. Svobodnyi 79
\\
         660041 Krasnoyarsk
\\
         Russia}

\email{paskaattt@yandex.ru}

%% Second author

\author{A. Shlapunov}

\address[Alexander Shlapunov]
        {Siberian Federal University
\\
         Institute of Mathematics and Computer Science
\\
         pr. Svobodnyi 79
\\
         660041 Krasnoyarsk
\\
         Russia}

\email{ashlapunov@sfu-kras.ru}

\date{June 11, 2014}



\begin{abstract}
We consider a (generally, non-coercive) mixed boundary value problem in a bounded 
domain $D$ of ${\mathbb R}^n$ for a second order parameter-dependent elliptic differential 
operator $A (x,\partial, \lambda)$ with complex-valued essentially bounded measured 
coefficients and complex parameter $\lambda$. The differential operator is assumed to be of 
divergent form in $D$, the boundary operator $B (x,\partial)$ is of Robin type with possible 
pseudo-differential components on $\partial D$. The boundary of $D$ is  assumed to be a 
Lipschitz surface. Under these assumptions the pair $(A (x,\partial, \lambda),B)$ 
induces a  holomorphic family of  Fredholm operators $L(\lambda): H^+(D) \to H^- (D)$ in 
suitable Hilbert spaces $H^+(D)$ ,  $H^- (D)$  of Sobolev type. If the argument of the 
complex-valued multiplier of the parame\-ter in $A (x,\partial, \lambda)$ is continuous and the 
coefficients related to second order derivatives of the operator are smooth then we prove that 
the operators $L(\lambda)$ are conti\-nu\-ously invertible for all $\lambda$ with sufficiently 
large modulus $|\lambda|$ on each ray on the complex plane $\mathbb C$ where the 
differential operator $A (x,\partial, \lambda)$ is parameter-dependent elliptic. We  also 
describe reasonable conditions for  the system of root functions related to the family 
$L (\lambda)$ to be (doubly) complete in the spaces $H^+(D)$, $H^- (D)$ 
and the Lebesgue space $L^2 (D)$. 
\end{abstract}

\maketitle

%\tableofcontents

\section*{Introduction}
\label{s.Intr}

The notion of a parameter-dependent elliptic operator  provides a useful 
link between the theories of boundary value problems for parabolic and elliptic operators 
(see, for instance, \cite{AgraVi64}).  Investigating a boundary  value problem 
for  parameter-dependent elliptic operator $A (x,\partial, \lambda)$ on a ray in the complex 
plane, first one aims to prove the continuous invertibility in proper functional spaces 
$H^+(D)$, $ H^- (D)$ of the corresponding family $L (\lambda): H^+(D) \to H^- (D)$  of 
the operators for  all $\lambda$ with sufficiently large modulus $|\lambda|$ on the ray (see 
\cite{AgraVi64},  \cite{DMV01}, \cite{DV02}, \cite{Pal96}). 
The next step is to prove the (multiple) completeness of the 
corresponding root functions associated with the parameter-dependent family (see for 
instance \cite{GokhKrei69}, \cite{Keld51}, \cite{Mar86}, 
\cite{Yak90}). Actually, this provides 
a justification for application of Galerkin type methods and numerical solution of the 
problem. For elliptic (coercive) problems the results of this type are well known. The 
investigation is usually based on the classical methods of functional analysis and the theory of 
partial differential equations  (see  \cite{Agmo62}, \cite{Brow53}, \cite{DunfSchw63}, 
\cite{GokhKrei69}, \cite{Keld51}, \cite{Mar86} and many others). 
For domains with smooth boundaries, 
the standard Shapiro-Lopatinsky conditions with parameter and their generalizations 
are usually imposed  (see \cite{AgraVi64}, \cite{DMV01}, \cite{DV02}). It is appropriate mention here that the spectral theory in non-smooth 
domains usually depends upon hard analysis near singularities on the boundary (see, for 
instance, \cite{Agra11a}, %\cite{EgorKondSchu01}, 
\cite{Tark06}). 

Recently the classical approach was adapted for investigation of spectral proper\-ties 
of non-coercive mixed problems for strongly  elliptic operators in Lipschitz domains (see 
\cite{PolkShla13}, \cite{ShlTark12}). An essential part of the approach is the analysis 
in spaces of negative smoothness. We successfully apply this method for studying non-
coercive boundary value problems for the parameter-dependent elliptic operators with 
complex coefficients in Lipschitz domains in the case where the argument of 
the complex-valued multiplier of the parameter in $A (x,\partial, \lambda)$ is continuous 
and the coefficients related to second order derivatives of the operator are smooth.

An example related to a non-coercive mixed problems for strongly elliptic 
two-dimensional Lam\'e system is considered.

\section{A Fredholm holomorphic family of mixed problems}
\label{s.FrFam}

Let $D$ be a bounded domain in Euclidean space ${\mathbb R}^n$ with Lipschitz boundary 
$\partial D$, i.e. the surface $\partial D$ is locally the graph of a Lipschitz function.

We consider complex-valued functions defined in the domain $D$. We write $L^q (D)$ for 
the space of all (equivalence classes of) measurable functions $u$ in $D$, such that the 
Lebesgue integral of $|u|^q$ over $D$ is finite. 
%When endowed with the norm $  \| u \|_{L^q (D)} = \Big( \int_D |u|^q\, dx \Big)^{1/q}$, the space $L^q (D)$ is Banach.
As usual, this scale continues to include the case $q = \infty$, too.
As usual, we denote by $H^{1} (D)$ the
Sobolev  space and by $H^s (D)$, $0 <s<1$ the Sobolev-Slobodetskii spaces.
% with the scalar product 
%$$
%( u,v)_{H^1 (D)} =  \int_{D} \sum_{j=1}^n \overline{\partial_j v} \, \partial_j u \, dx  + 
%\int_{D} \overline{v}  u \, dx   ; 
%$$
%here $x = (x_1, \dots , x_n)$  are the coordinates in 
%${\mathbb R}^n$ and $\partial_j = \frac{\partial }{\partial x_j}$. 
%Given any positive non-integer $0<s<1$, 
%we denote by $H^{s} (D)$ the Sobolev-Slobodetskii space with the scalar product 
%$$
%  ( u,v)_{H^s (D)}
% =  ( u ,v)_{L^{2} (D)}
%       + \int \!\!\!\!\! \int\limits _{D \times D}
%\!\!                \frac{\overline{(u (x) -  u (y))} 
%( u (x) -  u (y))}{|x - y|^{n+2s}}\, dx dy. 
%$$

Consider a second order differential operator 
$$
   A  (x,\partial, \lambda) u 
 = - \sum_{i, j = 1}^n \partial_i (a_{i,j} (x) \partial _j u)
   + \sum_{j = 1}^n a_j (x) \partial_j u
   + a_0 (x) u + E ( \lambda) u
$$
of divergence form in the domain $D$ with a complex parameter $\lambda$; 
here $x = (x_1, \dots , x_n)$  are the coordinates in 
${\mathbb R}^n$ and $\partial_j = \frac{\partial }{\partial x_j}$ and 
$$
 E (\lambda) u = 
\lambda  \left( \sum_{j = 1}^n a^{(1)}_j  (x) \partial_j u +  a_0^{(1)} (x) u\right) + 
\lambda ^2 a_0^{(2)} (x) u . 
$$
The coefficients $a_{i,j}$, $a_j$, $a_j ^{(1)}$, $a_0 ^{(1)}$, 
$a_0 ^{(2)} $ are assumed to be complex-valued functions of class $L^\infty (D)$.

We suppose that the matrix
$ {\mathfrak A} (x) =  \left( a_{i,j} (x) \right)_{\substack{i = 1, \ldots, n \\
                                         j = 1, \ldots, n}}$
is Hermitian and satisfies
\begin{equation}
\label{eq.ell.positive}
   \sum_{i,j=1}^n a_{i,j} (x) \overline{w}_i w_j
 \geq 0 \mbox{ for all } (x,w) \in \overline{D} \times \C^n, 
\end{equation}
\begin{equation}
\label{eq.ell}
   \sum_{i,j=1}^n a_{i,j} (x) \xi_i \xi_j
 \geq
   m_0\, 
	|\xi|^2 \mbox{ for all } (x,\xi) \in \overline{D} \times ({\mathbb R}^n \setminus \{ 0 \}),
\end{equation}
where $m_0$ is a positive constant independent of $x$ and $\xi$. Estimate (\ref{eq.ell}) is 
nothing but the statement that the operator $A (x,\partial,0)$ is strongly elliptic. It should be 
noted that,   since the coefficients of the operator and the functions under consideration are
complex-valued, the matrix ${\mathfrak A} (x)$ can be degenerate. In particular inequalities 
(\ref{eq.ell.positive}) and (\ref{eq.ell}) are weaker than the (strong) coerciveness of the 
Hermitian form, i.e. the existence of a constant $m_0$ such that
\begin{equation}
\label{eq.coercive.strong}
   \sum_{i,j=1}^n a_{i,j} (x)\, \overline{w}_i w_j
 \geq
   m_0\, |w|^2
\end{equation}
for all $(x,w) \in \overline{D} \times ({\mathbb C}^n \setminus \{ 0 \})$.

We consider the following Robin type boundary operator  
$$
B = b_1 (x) \sum_{i,j=1}^n a_{i,j} (x)\, \nu_i \partial_j + \partial_\tau +  B_0  
$$
where  $ b_1$ is a bounded function  on $\partial D$, $\nu (x) = (\nu_1 (x), \ldots, \nu_n (x))$
is the unit outward normal vector of $\partial D$ at $x \in \partial D$, $\partial_\tau = 
\sum_{j=1}^n \tau_j (x)\partial_j$ is the tangential derivative  with a tangential field $\tau = 
(\tau_	1, \dots, \tau_n)$ on $\partial D$ and  $B_0$  is a densely defined linear operator in
$L^2 (\partial D)$ of ``order'' does not exceeding $1$. The function $b_1 (x)$ is allowed 
to vanish on an open connected subset $S$ of $\partial D$ with piecewise smooth boundary 
$\partial S$ and the vector $\tau$ vanishes identically on $S$. 

To specify the operator $B_0$, fix a number  $0 \leq \rho \leq 1/2$ and a
bounded linear operator $\Psi : H^{\rho} (\partial D) \to L^{2} (\partial D)$.
The range of $\rho$ is motivated by trace and duality arguments.
We will consider operator $B_0$ of the following form 
$$
B_0 = \chi_S u + b_1 \left(\iPsi^\ast \iPsi (u) +  \delta B_0 \right) 
$$
where $\chi_S$ is the characteristic function of the
set $S$ on $\partial D$, 
$\Psi ^*: L^{2} (\partial D) \to H^{\rho} (\partial D)$ 
is the adjoint operator for $\Psi$  and $\delta B_0$ is a ``low order'' perturbation that 
we will describe later.

For $\rho=0$ a typical operator $\Psi$ is a zero order differential operator, 
i.e. it is given by  $\Psi u = \psi u$,
where $\psi$ is a function on $\partial D$ locally bounded away from $\partial S$.
Then $(\Psi^\ast \Psi u) (x)= |\psi (x)|^2 u (x)$ is invertible provided that 
$|\psi (x)| \geq c > 0$. If $\partial D$ is $C^2\,$-smooth then a model operator $\Psi$ is
   $\Psi = (1 + \iD_{\partial D})^{\rho/2}$
where $\iD_{\partial D}$ is the Laplace-Beltrami operator on the boundary.

Consider the following family of boundary value problems. 
Given data  $f$ in $D$ and $u_0$ on $\partial D$, find a distribution $u$ in $D$
which satisfies
\begin{equation}
\label{eq.SL}
   \left\{
\begin{array}{rclcl}
     A (x,\partial, \lambda) u
   & =
   & f
   & \mbox{in}
   & D,
\\
     B (x,\partial) u
   & =
   & u_0
   & \mbox{at}
   & \partial D.
\end{array}
   \right.
\end{equation}
 
If $\lambda=0$ and $\Psi$ is given by the multiplication on a function, this is a well known 
mixed problem of Zaremba type (see \cite{Zare10}). It can be handled in a standard way in 
Sobolev type spaces associated  with Hermitian forms or in H\"older spaces and Sobolev 
spaces using the potential methods, (for the coercive case see \cite{Zare10}, \cite{LiMa72}, 
\cite{LadyUral73}, \cite{Gris85} and elsewhere). In the non-coercive case the methods  
should be more subtle (see, for instance, \cite{ADN59}, \cite{ShlTark12}) because of the 
lack of regularity of its solutions  near the boundary of the domain. 

In \cite{ShlTark12} the method, involving  non-negative  Hermitian forms,  was adopted 
to study problem (\ref{eq.SL}) in non-coercive cases with a zero order differential operator 
$\Psi$. Namely,  denote by $C^1 (\overline D,S)$ the subspace of $C^1 (\overline D)$ 
consisting of those functions whose restriction to the boundary vanishes on $\overline S$. Let 
$H^1 (D,S)$ be the closure of $C^1 (\overline D,S)$ in $H^1 (D)$. This space is Hilbert 
under the induced norm. Since on $S$ the boundary operator reduces to $B = \chi_S$ and $
\chi_S (x) \neq 0$ for $x \in S$, the functions of $H^1 (D)$ satisfying $Bu = 0$ on $\partial 
D$ belong to $H^1 (D,S)$. 

Split $a_0$ into two parts
$
 a_0 = a_{0,0} + \delta a_0,
$ 
where    $a_{0,0}$ is a non-negative bounded function in $D$.  
Then, under reasonable assumptions, the Hermitian form
$$
   (u,v)_{+}
 =  \int_{D} \sum_{i,j=1}^n a_{i,j} \partial_j u \overline{\partial_i v}\, dx
 + (a_{0,0} u, v)_{L^2 (D)}
 +  ( \Psi (u), \Psi (v) )_{L^{2} (\partial D)}
$$
defines a scalar product on  $H^1 (D,S)$. 
Denote by $H^{+} (D)$ the completion of the space $H^1 (D,S)$
with respect to the corresponding norm $\|\cdot\|_+$. 

To study the problem (\ref{eq.SL}) we need an embedding theorem for the space $H^+ (D)$.

\begin{thm}
\label{t.emb.half}
Let the coefficients $a_{i,j}$ be $C^\infty$ in a neighbourhood $X$ of the closure of $D$, 
inequalities (\ref{eq.ell.positive}), (\ref{eq.ell}) hold and
there is a constant $c _1> 0$, such that
\begin{equation}
\label{eq.b}
   \| \Psi u \|_{L^2 (\partial D)}
 \geq
   c_1\, \| u \|_{H^\rho (\partial D)} \mbox{ for all } u \in H^1 (\partial D, S).
\end{equation}
If there is a positive constant $c_2$, such that
\begin{equation}
\label{eq.a}
   a_{0,0} \geq c_2 \mbox{ in } D
\end{equation}
or  the operator $A$ is strongly elliptic in a neighbourhood $X$ of $\overline D$
and
\begin{equation}
\label{eq.aa}
   \int_{X} \sum_{i,j=1}^n a_{i,j} \partial_j u \overline{\partial_i u}\, dx
 \geq
   m_1\, \| u \|^2_{L^2 (X)}
\end{equation}
for all $u \in C^\infty_{\comp} (X)$, with $m_1 > 0$ a constant independent of $u$ 
then the space $H^+ (D)$ is continuously embedded into $H^{s} (D)$ where $s$  
is given by
\begin{equation}
\label{eq.r+half}
   s= \left\{
   \begin{array}{lll}
   1/2-\epsilon\ \mbox{with}\ \epsilon > 0,
 & \mbox{if}
 & \rho=0,
\\
   1/2,
 & \mbox{if}
 & \rho=0\ \ \mbox{and}\ \ \partial D \in C^2,
\\
   1/2+\rho,
 & \mbox{if}
 & 0 < \rho \leq 1/2.
\end{array}
   \right.
\end{equation}
\end{thm}

\begin{proof} %See \cite[Theorem 8.4]{ShlTark13}. Actually, 
It  is similar to the proof of 
\cite[Theorem 2.5]{ShlTark12} corresponding to the case where $\rho=0$ and 
$\Psi$ is given by the multiplication on a function. 
\end{proof}

Of course, under the coercive estimate (\ref{eq.coercive.strong}), the space $H^+ (D)$ is  
continuously em\-bed\-ded into $H^{1} (D)$. However, in general, the embedding, described 
in Theorem \ref{t.emb.half} is rather sharp (see \cite[Remark  5.1]{ShlTark12}. and 
\S \ref{s.exmp}  below). In particular, if $\rho=0$ then it may happens that the space 
$H^+ (D)$ can not be embedded into $H^{1/2+\epsilon} (D)$ with any $\epsilon>0$. 
Thus the operator $\Psi$ is introduced here in order to improve, if necessary, the smoothness 
of elements of $H^+ (D)$ in the non-coercive case.

In order to pass to the generalized setting of the mixed problem we need 
that all the derivatives $\partial_j u$ belong to $L^2(D)$ for an element 
$u \in H^+ (D)$, at least if $s\leq 1/2$ in Theorem \ref{t.emb.half}. 
However if $0<s<1$  then the absence of 
coerciveness does not allow this. To cope with this difficulty we note that 
the operator $ \sum_{i,j=1}^n \partial _i (a_{i,j} \partial_j \cdot )$ admits a 
factorisation, i.e. there is an $(m \times n)\,$-matrix ${\mathfrak D} (x)  = 
\left( {\mathfrak D}_{i,j} (x) \right)_{\substack{i = 1, \ldots, m \\
                                         j = 1, \ldots, n}}$ of bounded
functions in $D$, such that
\begin{equation}
\label{eq.factor}
 ({\mathfrak D} (x))^\ast {\mathfrak D} (x)  =  {\mathfrak A} (x) 
\end{equation}
for almost all $x \in D$. For example, one 
could take the standard  non-negative selfadjoint square root  
${\mathfrak D} (x)= \sqrt{{\mathfrak A} (x)}$ of the matrix ${\mathfrak A} (x)$. 
Then 
$$
\sum_{i,j=1}^n a_{i,j} \partial_j u \overline{\partial_i v}\ 
 =  ({\mathfrak D} \nabla v ) ^* {\mathfrak D}  \nabla u  = 
 \sum_{k=1}^m \overline{{\mathfrak D}_k v}\, {\mathfrak D}_k u,
$$
for all smooth functions $u$ and $v$ in $\overline D$, where
   $\nabla u$ is thought of as $n\,$-column with entries
   $\partial_1 u, \ldots, \partial_n u$,
and $   {\mathfrak D}_k u:= \sum_{l=1}^n {\mathfrak D}_{k,l} (x) \partial_l u$, 
$k = 1, \ldots, m$.
Then, by the definition of the space $H^+(D)$, 
any term $\tilde a_k (x) {\mathfrak D}_k u$, $k = 1, \ldots, m$, belongs to 
$L^2 (D)$ if $u \in H^+(D)$ and $\tilde a_k \in L^\infty (D)$.
Thus, if $0<s<1$ then we may confine ourselves 
with first order summands of the form
 $$ 
\sum_{k=1}^m \tilde{a}_k (x) {\mathfrak D}_k \mbox{ and } \sum_{k=1}^m 
\tilde{a}^{(1)}_k (x) {\mathfrak D}_k
$$
instead of
$   
\sum_{j=1}^n a_{j} (x) \partial_j $ %\mbox{ 
and %} 
$\sum_{k=1}^n a^{(1)}_j (x) \partial_j$.  
For this purpose, we fix a factorization ${\mathfrak D} (x) $ of the matrix
 ${\mathfrak A} (x) $ and functions $\tilde a_k \in L^\infty (D)$, $\tilde a^{(1)}_k \in 
L^\infty (D)$, $k = 1, \ldots, m$.

These considerations allow to handle problem (\ref{eq.SL}) with the use of the standard tools 
of functional analysis. Indeed, let  $H^{-} (D)$ stand  for the completion of space 
$H^{+} (D)$ with respect to the norm 
$$
   \| u \|_{-}
 = \sup_{\substack{v \in H^{+}   (D) \\ v \ne 0}}
   \frac{|(v,u)_{L^2 (D)}|}{\| v \|_{+}}.
$$
It is the dual space for the space $H^{+} (D)$ with respect to the pairing 
$$
\langle  \cdot,\cdot \rangle : H^{-} (D) \times H^{+} (D) \to {\mathbb C} 
$$ 
induced by the scalar product $(\cdot,\cdot)_{L^2 (D)}$, 
$$
\langle  u, v \rangle = \lim_{\nu \to + \infty} (u_\nu,v)_{L^2 (D)}, \quad 
u \in H^{-} (D), v \in H^{+} (D)
$$ 
where $\{ u_\nu \} \subset H^{+} (D)$ converges to $u$ in $H^{-} (D)$, see 
\cite{LiMa72}. Note that under hypothe\-sis of Theorem \ref{t.emb.half}, the natural 
embedding $\iota : H^{+} (D) \to L^2 (D)$ is continuous; it is compact if (\ref{eq.b}) holds. 
Let $\iota ^\prime:  L^2 (D) \to H^{+} (D) $ stand for the adjoint map for $\iota$ with 
respect to the pairing $\langle \cdot , \cdot \rangle$, 
i.e. 
$$
\langle  \iota ^\prime u , v \rangle = (u, \iota v)_{L^2 (D)}
\mbox{ for all } u \in L^{2} (D), v \in H^{+} (D). 
$$

Now an integration by parts leads to a weak formulation of problem (\ref{eq.SL}):
given $f \in H^- (D)$, find $u \in H^+ (D)$, such that
\begin{equation}
\label{eq.SL.w}
  (u,v)_+ + \Big(\big(\sum_{j=k}^m \tilde a_k {\mathfrak D} _k  + \delta a_0 + 
E(\lambda) \big)u ,v  \Big)_{L^2 (D)} 	 +  \left( (b_1^{-1} \partial_\tau + 
	\delta B_0) u, v \right)_{L^2 (\partial D \setminus S)} = <f,v>
\end{equation}
for all $v \in C^1 (\overline D, S)$.

By the Cauchy inequality, if 
$$
\left| \left( (b_1^{-1} \partial_\tau + 
	\delta B_0) u, v \right)_{L^2 (\partial D \setminus S)} \right| \leq c \, \|u\|_+ \, \|u\|_+
$$
with a constant $c>0$ independent on  $u,v\in H^{+} (D)$, then 
(\ref{eq.SL.w}) induces a holomorphic (with respect to 
$\lambda \in {\mathbb C}$) family $L(\lambda): H^{+} (D) \to H^{-} (D) $ 
 of bounded linear operators. 

Denote by $L_0$ the operator $L(0)$ in the case 
where $\tau \equiv 0$, 
$\delta B_0\equiv 0$, $\delta a_0\equiv 0$, $\tilde a_k \equiv 0$, $k=1,\dots m$.
According to \cite[Lemma 2.6]{ShlTark12}, the operator $L_0: H^{+} (D) \to H^{-} (D) $ is 
continuously invertible and $\|L_0\|=\|L_0^ {-1}\|=1$. Then we can 
consider each operator $L(\lambda)$, $\lambda \in {\mathbb C}$, as a perturbation of $L_0$.

Actually, it is convenient to endow the space $H^{-} (D)$ with the 
scalar product 
\begin{equation} \label{eq.-}
(u,v)_ -= (L_0^{-1} u, L_0^{-1} v )_+ = <L_0^{-1} u ,v>, \, u,v \in H^- (D)
\end{equation}
coherent with the norm $\|\cdot\|_-$ see, for instance  \cite[p. 3316 and 
formula (2.2)]{ShlTark12}. 

We can provide more subtle properties of the family $\{ L(\lambda) \} 
_{\lambda \in {\mathbb C}}$ under reasonable assumptions.

In the sequel ${\mathcal L} (H_1, H_2)$ stand for the space of 
bounded linear operators in Banach spaces $H_1$ and $H_2$.

\begin{lem}\label{l.Fred.non}
Under the hypothesis of Theorem \ref{t.emb.half}, let $0\leq  \rho \leq 1/2$. 
If $  \delta B_0 $ maps  $H^{\rho} (\partial D, S)$ continuously 
into $H^{-\rho} (\partial D)$ then the term $( \delta B_0 u,v)_{L^2 (\partial D)}$ 
indu\-ces a bounded operator 
$\delta L_B: H^+ (D) \to H^{-} (D)$ and 
$$
\|\delta L_B\|_{{\mathcal L} (H^+ (D), H^{-} (D))} 
\leq  \|  \delta B_0 \| _{{\mathcal L}(H^\rho (\partial D, S), H^{-\rho} (\partial D))}
 \|\Psi^{-1}\|^2.
$$ 
If $ \delta B_0 $ maps $H^{\rho} (\partial D, S)$ compactly 
into $H^{-\rho} (\partial D)$ then the operator $\delta 
L_B$ is compact. In particular, if $\delta B_0 $ 
is given by the multiplication on a function $\delta b_0 \in 
L^\infty (\partial D \setminus S)$ then 

1) $\delta B_0$ maps $H^{\rho} (\partial D, S)$ compactly into 
$H^{-\rho} (\partial D)$ for $0<\rho\leq 1/2$, 

2) $ \delta B_0$ maps $L^{2} (\partial D, S)$ continuously into 
$L^{2} (\partial D)$ for $\rho = 0$. 
\end{lem}

\begin{proof} The proof is standard, cf. \cite[Lemma 4.6]{ShlTark12}.
\end{proof}

Clearly, the linear span of the vectors
\begin{equation*} 
\tau_{i,j}=\vec{e}_j \nu_i (x) - \vec{e}_i \nu_j (x), \quad i>j.
\end{equation*}
coincides with the tangential plan at each point $x \in \partial D$ where it exists.
Thus, without loss of generality, we may consider tangential partial differential operators of
the following form:
$$
   \partial_\tau = \sum_{i>j} k_{i,j} (x)\partial_{\tau _{i,j}}
$$

\begin{lem}\label{l.Fred.coercive}
Let $H^{+} (D)$ be continuously embedded  into $H^1 (D,S)$. If 
   $k_{i,j} / b_1$ is of H\"older class $C^{0,\lambda}$ in the closure of 
   $\partial D \setminus S$ for all $i>j$, with $\lambda > 1/2$,
then the ``tangential''
 term $(b_1^{-1} \partial_\tau u, v)_{L^2 (\partial D \setminus S)}$
induces a bounded operator $\delta L_\tau: H^{+} (D) \to H^{-} (D)$.
\end{lem}
\begin{proof} The statement was proved in \cite[Lemma 6.6]{ShlTark12}.
\end{proof}

\begin{thm} \label{t.hol.fam} Under the hypothesis of Theorem \ref{t.emb.half}, 
let $\tau=0$ unless $s=1$. If  either the term $( \delta B_0 u, v)
_{L^2 (\partial D)}$ 
induces a bounded operator $\delta L_B$ from  $H^+ (D)$ to $H^-(D)$ with 
$\| \delta L_B   +  \delta L_\tau\| < 1$ or $\|  \delta L_\tau\| < 1$ and the term 
$(  \delta B_0 u, v)_{L^2 (\partial D)}$  induces a compact 
operator from $H^+ (D)$ to $H^-(D)$ 
then $\{ L(\lambda) \} _{\lambda \in {\mathbb C}}$ is a holomor\-phic 
family of Fredholm operators of zero index. 
\end{thm}

\begin{proof} Follows from Lemmas \ref{l.Fred.non} and \ref{l.Fred.coercive} 
because $H^{+} (D) $ is compactly embedded into $L^2 (D)$ under  
hypothesis of Theorem \ref{t.emb.half}. 
%Indeed, the operator 
%$$
%sum_{k=1}^m (\tilde a_k   + \lambda \tilde a^{(1)}_k) {\mathfrak D} _k 
% \delta a_0 +  \lambda a^{(1)}_0 + \lambda^2 a^{(2)}_0
%$$
% maps $H^{+} (D) $ continuously into 
%$L^2 (D)$ because $\tilde a_k, \tilde a^{(1)}_k  \delta a_0 ,   
%a^{(1)}_0,   a^{(2)}_0 \in L^\infty (D)$.  
%On the other hand, Rellich Theorem and Theorem \ref{t.emb.half} imply that 
%$H^{+} (D) $ is compactly embedded into $L^2 (D)$. 
%Therefore , according to \cite[Lemma 2.2]{ShlTark12}), 
%the operator $\delta L_c: H^{+} (D) \to H^{-} (D)$, induced by the term 
%$$
% \Big(\big(\sum_{k=1}^m (\tilde a_k   + \lambda \tilde a^{(1)}_k) {\mathfrak D} _k 
%+ \delta a_0 +  \lambda a^{(1)}_0 + \lambda^2 a^{(2)}_0\big)u ,v  \Big)_{L^2 (D)} 
%$$ 
%in the weak formulation (\ref{eq.SL.w}), is compact. 
%
%f the term $(  \delta B_0 u, v)_{L^2 (\partial D)}$ induces a bounded operator from  
%H^+ (D)$ to $H^-(D)$ with $\| \delta L_B  +  \delta L_\tau\| < 1$ then, by the familiar 
%rguments, the operator $L_0 + \delta L_B   +  \delta L_\tau $ is continuously invertible. 
%Hence the polynomial family $L (\lambda) = L_0 + \delta L_B   +  \delta L_\tau + 
%\delta L_c$ consists of Fredholm operators of zero index for each $\lambda \in \mathbb 
%C$. %The proof of the second case is similar.
\end{proof}

\section{Mixed problems for parameter-dependent elliptic operators}
\label{s.PE}

To obtain the main theorem of this paper  
we invoke the notion of parameter-dependent ellipticity. 

We recall that  the operator $A  (x,\partial, \lambda) $ is parameter-dependent elliptic on a  
ray $\Gamma = \{\mbox{arg} (\lambda) =\varphi_\Gamma \}$ on the complex plane 
$\mathbb C$ if 
\begin{equation} \label{eq.PE}
    \sum_{i, j = 1}^n a_{i,j} (x)  \zeta _i \zeta _j  
+\lambda  \sum_{j = 1}^n a_j^{(1)} (x) \zeta_j 
+ \lambda ^2 a_0^{(2)} (x) \ne 0
\end{equation}
for all $x \in \overline D$ and all $(\lambda,\zeta) \in (\Gamma \times {\mathbb R}^n) 
\setminus \{0,0\}$. 

In particular, if the operator $A  (x,\partial, \lambda)  $ is parameter-dependent elliptic on 
the ray $\Gamma $ then taking $\zeta=0$ and $\lambda \ne 0$ in (\ref{eq.PE}) we obtain 
$ a_{0}^{(2)} (x) \ne 0$ for all $x \in D$.

In the sequel we consider the case where $E (\lambda)=\lambda ^2 a_0^{(2)} (x)$, 
the most common in applications. Then we prove that, under reasonable assumptions, 
the family $L (\lambda): H^+(D) \to H^- (D)$ is 
the continuously invertible for all $\lambda$ with sufficiently large modulus $|\lambda|$ 
on the ray $\Gamma$ where the operator $A  (x,\partial, \lambda)  $ is 
parameter-dependent elliptic (cf.  \cite{AgraVi64},  \cite{Pal96}).

Let  $\varphi_0 (x) = \arg{\left(a^{(2)}_0 (x)\right)}$.
Denote by  $C: H^{+} (D) \to H^{-} (D)$ the operator that is induced by the 
term $(a^{(2)}_0 (x) u, v)_{L^2 (D)}$.

\begin{lem}\label{l.inj.C}
Let $ a_{0}^{(2)} (x) \ne 0$ for almost all $x \in D$.
Then the operator $C: H^{+}   (D) \to H^{-}   (D)$ is injective. 
\end{lem}

\begin{proof} Indeed, if $Cu=0$ then 
$$
0=\langle Cu, v \rangle =\int_{D} a_{0}^{(2)} (x) u (x) \overline v (x)  \ dx 
\mbox{ for all } v \in 
H^{+}   (D).
$$
As the $H^{+}   (D)$ is dense in $L^2 (D)$ we see that 
$a_{0}^{(2)} u =0$ almost everywhere in $D$. 

Finally, as $ a_{0}^{(2)} (x) \ne 0$ for almost all $x \in D$ we conclude that 
$u =0$ almost everywhere in $D$. 
\end{proof}

\begin{lem} \label{l.PE.cond}
Suppose that the matrix ${\mathfrak A} (x)$ is Hermitian non-negative. If $E(\lambda) = 
\lambda^2 a_0^{(2)}$ then the operator $A (x,\partial ,\lambda)$ is parameter-dependent 
elliptic  on the ray $\Gamma$ if and only if 
\begin{equation}
\label{eq.ell.w}
  \sum_{i,j=1}^n a_{i,j} (x) \xi_i \xi_j>0
\end{equation}
for all $(x,\xi) \in \overline{D} \times ({\mathbb R}^n \setminus \{ 0 \})$, 
\begin{equation} \label{eq.PE.cond1}
|a^{(2)}_0 (x) |  >0 \mbox{ for all } x \in \overline D;
\end{equation}
\begin{equation} \label{eq.PE.cond2}
\cos{(\varphi_0 (x) + 2\varphi _\Gamma)}>-1 \mbox{ for all }  x \in \overline D. 
\end{equation}
\end{lem}
\begin{proof} Follows from the standard trigonometrical formulas. 
\end{proof}

Of course, if $a_{i,j} \in C (\overline D)$ then (\ref{eq.ell.w}) is equivalent 
to (\ref{eq.ell}). If $|a^{(2)}_0 (x)| \in C (\overline D)$ then (\ref{eq.PE.cond1}) is 
equivalent to the following
\begin{equation} \label{eq.PE1}
|a^{(2)}_0 (x) | \geq   \theta_0>0 \mbox{ for all } x \in \overline D;
\end{equation}
similarly, if $\varphi_0 (x)  \in C (\overline D)$ 
then  (\ref{eq.PE.cond2}) is equivalent 
to the following
\begin{equation} \label{eq.PE2}
\cos{(\varphi_0 (x) + 2\varphi _\Gamma)}\geq \theta_1 (\Gamma) =\theta_1 >-1 \mbox{ for all } x \in \overline D, 
\end{equation}
where the constants $\theta_0$, $\theta_1$ do not depend on $x$.

Clearly, under the hypothesis of Theorem \ref{t.hol.fam}
we can decompose 
$$
L(\lambda) = L_0 + \delta_c L + \delta_s L + \lambda^2 C
$$
where $\delta_c L:H^+ (D) \to H^- (D)$ is a compact operator and 
$\delta_s L:H^+ (D) \to H^- (D)$ is a bounded one. 
Moreover, the family $L(\lambda)$ is Fredholm if $\|\delta_s L\|<1$.

Let
$
\eta (\Gamma) = \max{(0, -\theta_1)} .
$

\begin{thm}
\label{t.PE.inv}
\textcolor{red}{
Let either $\Psi$ is given by the multiplication on a function $\psi \in L^\infty (\partial D)$
or $\partial D \in C^\infty$ and $\Psi$ is a pseudodifferential operator on $\partial D$.
}
Let also $E(\lambda) = \lambda ^2 a^{(2)}_0 $, 
the hypothesis of Theorem \ref{t.hol.fam} be fulfilled, 
\begin{equation} \label{eq.PE0}
a^{(2)}_0 \ne 0 \mbox{ almost everywhere in } D
\end{equation} 
and (\ref{eq.PE2})  hold true. 
If $\varphi_0 \in C(\overline D)$  and 
$\|\delta_s L\|^2 +\eta^2 (\Gamma)<1$ then 
\begin{enumerate}
\item[{\rm 1)}]   
 there is $\gamma_0 \in \Gamma$
such that the operators $L(\lambda): H^{+}   (D)\to H^{-} (D)$ are  
continuously invertible for all $\lambda \in \Gamma$ with $|\lambda|\geq |\gamma_0|$;
\item[{\rm 2)}]
the operators $L(\lambda)$ are continuously invertible for all $\lambda \in {\mathbb C}$ 
except a discrete countable set $\{\lambda_\nu \}$ without limit points in ${\mathbb C}$. 
\end{enumerate}
\end{thm}

\begin{proof} We begin with the following lemma.

\begin{lem} \label{l.L0.PE.inv}
Under the hypothesis of Theorem \ref{t.PE.inv}, 
there is $k_0 \in \mathbb N$ such that for all $\lambda \in \Gamma$ 
with $|\lambda| \geq k_0$ we have 
$$
\|(L_0 +\delta_s L + \lambda^2 C )u\|_{-}\geq \left(\sqrt{1-\eta^2 (\Gamma)} 
- \|\delta _s L\| \right) \|u\|_{+} 
\mbox{ for all } u \in H^{+}   (D) 
$$
and there are positive constants $p_1 = p_1 (\varphi_\Gamma), q_1= 
q_1 (\varphi_\Gamma)$ such that 
\begin{equation} \label{eq.rays.minimal}
\|(L_0 + \delta_s L + \lambda^2 C )u\|_{-}\geq p_1 \|u\|_{+} + 
q_1 |\lambda|^2 \|C u\|_- 
\end{equation}
for all $ u \in H^{+}   (D) $  and $ \lambda\in \Gamma$ with $|\lambda| \geq k_0$.
\end{lem}
\begin{proof} 
Given any $u \in H^{+}   (D)$ 
an easy computation with the use of formula (\ref{eq.-}) 
shows that
\begin{equation} \label{eq.CC}
\lambda^2  \langle  C  u,   u\rangle  = 
| \lambda|^2  \int_{D}  |a^{(2)}_0 (x)| |u (x)|^2 
e^{{\sqrt{-1}}(\varphi_0 (x) + 2\varphi _\Gamma)} \ dx , 
\end{equation}
\begin{equation} \label{eq.L0.PE1.inv}
   \|(L_0 + \lambda ^2C) u \|_{-}^2
  =    \langle u + \lambda^2 L_0^{-1} C  u, (L_0 + \lambda ^2C) u \rangle^2 = 
\end{equation}
$$
   \langle u , L_0 u \rangle  + \langle  \lambda^2 L_0^{-1} C  u,  \lambda^2 C  u 
		\rangle + \overline \lambda^2 \langle u , C u \rangle 
+ \lambda^2 \langle   L_0^{-1} C  u,   L_0 u \rangle = 
$$
$$
   \| u \|_{+}^2 +| \lambda|^4  \| Cu \|_{-}^2 + 
\overline \lambda^2 \langle u , C u \rangle + \lambda^2 (L_0^{-1} C  u,   u)_{+}   = 
$$
$$
   \| u \|_{+}^2 +| \lambda|^4  \| Cu \|_{-}^2 + 
\overline \lambda^2 \langle u , C u \rangle
+ \lambda^2  \langle  C  u,   u\rangle  = 
$$
$$
   \| u \|_{+}^2 +| \lambda|^4  \| Cu \|_{-}^2 
+ 2\Re \Big(\lambda^2  \langle  C  u,   u\rangle  \Big).
$$
Clearly, for $\lambda \in \Gamma$, 
\begin{equation} \label{eq.PE.re.inv}
\Re \Big(\lambda^2  \langle  C  u,   u\rangle \Big) = 
| \lambda|^2  \int_{D}  |a^{(2)}_0 (x)| |u (x)|^2 
\cos{(\varphi_0 (x) + 2\varphi _\Gamma)} \ dx .
\end{equation}
If $\theta_1 \in [0,1]$ then $\eta (\Gamma)=0$ and we have immediately for all $u \in H^{+}   (D)$: 
$$
   \|(L_0 + \lambda ^2C) u \|_{-}^2   \geq   \| u \|_{+}^2 +| \lambda|^4  \| Cu \|_{-}^2, 
$$
$$
 \|(L_0 +\delta_s L+ \lambda ^2C) u \|_{-} \geq  \|(L_0 + \lambda ^2C) u \|_{-} - 
 \|\delta_s L u \|_{-} \geq 
$$
$$
\sqrt{ \| u \|_{+}^2 +| \lambda|^4  \| Cu \|_{-}^2} - \|\delta_s L u \|_{-}. 
$$
Then, for $\alpha \in [0,\pi/2]$ and non-negative numbers $a,b$, we have 
\begin{equation} \label{eq.alpha.inv}
\sqrt{a+b} \geq \sqrt{a} \cos{(\alpha}) + \sqrt{b} \sin{(\alpha}). 
\end{equation}
As $\|\delta_s L\|<\sqrt{1-\eta^2 (\Gamma)}=1$, there is $\alpha_0 \in (0,\pi/2)$ such that 
$$\|\delta_s L\|<\cos{(\alpha_0)}$$
In particular, this means that for all $u \in H^{+}   (D)$ and all $\lambda \in \Gamma$ we 
have: 
$$
\|(L_0 +\delta_s L+ \lambda ^2C) u \|_{-} \geq  \| u \|_{+} - 
 \|\delta_s L u \|_{-} \geq (1- \|\delta_s L \|) \|u\|_+, 
$$
$$
\|(L_0 +\delta_s L+ \lambda ^2C) u \|_{-} \geq  \cos{(\alpha_0)}\| u \|_{+} + \sin{(\alpha_0)}| \lambda|^2  
\| Cu \|_{-} - \|\delta_s L u \|_{-}  \geq 
$$
$$
(\cos{(\alpha_0)} - \|\delta_s L\|) \| u \|_{+} + \sin{(\alpha_0)}| \lambda|^2  
\| Cu \|_{-},
$$
i.e. the desired inequalities are true if $\theta_1 \in [0, 1]$.

If  $\theta_1 \in (-1, 0)$ then, by (\ref{eq.PE.re.inv}) and  (\ref{eq.PE2}),
\begin{equation} \label{eq.L0.PE2.inv}
\Re \Big(\lambda^2  \langle  C  u,   u\rangle \Big) \geq -|\theta_1|
| \lambda|^2  \int_{D}  |a^{(2)}_0 (x)| |u (x)|^2  \ dx.
\end{equation}

Let us prove that for any  $\theta \in (-\theta_1, 1]$ and $\gamma \in [0,1)$ 
with $\theta \sqrt{1-\gamma}>-\theta_1$ there is $k_0\in \mathbb N$ 
such that 
\begin{equation} \label{eq.ray.theta}
\|(L_0 + \lambda^2 C) u\|^2 \geq 
\left(1- \theta^2 \right)\|u\|^2_+ + \gamma |\lambda|^4 \|Cu \|^2_- 
\end{equation}
 for all $ u \in H^+ (D)$ and all $\lambda \in \Gamma$ with $|\lambda|\geq k_0$.
Indeed, we argue by contradiction. Let there are $\theta \in (|\theta_1|, 1]$ and 
$\gamma \in [0,1)$
with $\theta \sqrt{1-\gamma}>|\theta_1|$ such that for  each $k \in \mathbb N$  there are 
$u_k \in H^+ (D)$ with $\|u_k\|_+ =1$, and a number 
$\lambda_k \in \Gamma$ with $|\lambda_k| \geq k$ such that
$$
\|(L_0 + \lambda^2_k C) u_k\|^2 < 1-\theta^2  + \gamma  |\lambda_k|^4 \|Cu_k \|^2_-.
$$

It follows from (\ref{eq.L0.PE1.inv}) and  (\ref{eq.PE.re.inv}) 
that 
$$
\theta^2  + |\lambda_k|^4 \|Cu_k \|^2_- (1-\gamma) + 
2  | \lambda_k|^2  \int_{D} \cos{(\varphi_0 +2 \varphi_\Gamma)} 
|a^{(2)}_0 (x)| |u_k (x)|^2  \ dx <0,
$$
i.e. 
\begin{equation} \label{eq.main.inv}
\left(\theta  - \sqrt{(1-\gamma)} |\lambda_k|^2 \|Cu_k \|_-
\right)^2 + 
\end{equation}
$$
2  \left( \theta \sqrt{(1-\gamma)}  + \frac{
  \int_{D}  \cos{(\varphi_0 +2 \varphi_\Gamma)} 
|a^{(2)}_0 (x)| |u_k (x)|^2  \ dx }{\|Cu_k \|_-} \right) | \lambda_k|^2 \|Cu_k \|_-  <0,
$$
for all $k \in \mathbb N$.

On the other hand, for all $u \in H^+ (D) $ with $\|u\|_+ =1$ we have
$$
\|Cu\|_- = \| e^{2\sqrt{-1}\varphi_\Gamma} Cu\|_-  
\geq \left| ( e^{\sqrt{-1}(\varphi_0 + 2\varphi_\Gamma)} |a_0^{(2)}| u,u)_{L^2 (D)} \right|.
$$
In particular, we have
$$
\left| \frac{\int_{D}  \cos{(\varphi_0 + 2 \varphi_\Gamma)} 
|a^{(2)}_0 (x)| |u_k (x)|^2  \ dx }{\|Cu_k \|_-} \right| \leq 1
\mbox{ for all } k ´\in \mathbb N.
$$
Now, if the sequence $\{ | \lambda_k|^2 \|Cu_k \|_-\}$ is unbounded 
then extracting a subsequence $\{ | \lambda_{k_j}|^2 \|Cu_{k_j} \|_-\}$ 
tending to $+\infty$, dividing (\ref{eq.main.inv}) by $| \lambda_{k_j}|^4 \|Cu_{k_j} \|_-^2$ 
and passing to the limit with respect to $k_j \to +\infty $ we obtain $1\leq 0$, a contradiction. 

Let the sequence $\{ | \lambda_k|^2 \|Cu_k \|_-\}$ be bounded. 
%As $L_0$ is bounded it maps the bounded sequence $\{ u_k\}$ to 
%the bounded sequence $\{ L_0 u_k\} $.
Now the weak compact\-ness principle for Hilbert spaces yields that there is a subsequence 
$\{ u_{k_j} \}$ weakly convergent 
to an element $u_0$ in the space $H^+ (D)$. Then $\{ C u_{k_j} \}$ 
converges to $Cu_0$ in $H^{-} (D)$ because $C: H^+ (D) \to H^- (D)$ is compact 
and $\{ u_{k_j} \}$ converges to $u_0$ in $L^{2} (D)$ because the 
embedding $\iota: H^+ (D) \to L^2 (D) $ is compact, too. Since 
 the sequence $\{\lambda^2_{k_j} C u_{k_j} \}$
is bounded in $H^{-} (D)$ and $|\lambda_k| \to + \infty$ 
we conclude that $\{ C u_{k_j} \}$ converges 
to zero in $H^- (D)$. This means that $Cu_0=0$ and then $u_0=0$ because 
$a_0^{(2)} (x) \ne 0$ if (\ref{eq.PE0}) is fulfilled 
on the ray $\Gamma$ and then the operator $C$ is injective  
(see Lemma \ref{l.inj.C}). 

According to compactness principle, we may consider the subsequences  
$$
\left\{ |\lambda_{k_j}|^2 \|Cu_{k_j} \|_- \right\} \mbox{ and }
\left\{ - \frac{\int_{D}   \cos{(\varphi_0 +2 \varphi_\Gamma)} 
|a^{(2)}_0 (x)| |u_{k_j} (x)|^2  \ dx }{\|Cu_{k_j} \|_-} \right\} 
$$
as convergent to the limits $\alpha \geq 0$ and $\beta \in [-1,1]$ respectively.
Now it follows from (\ref{eq.main.inv}) that 
\begin{equation} \label{eq.contra}
\left(\theta  - \alpha \right)^2 + 2\alpha \left( \theta   - \beta \right) \leq 0.
\end{equation}
If $\alpha=0$ then we have a contradiction because $\theta >0$.
If $\alpha >0$ and $\beta \leq 0$  then $\theta -\beta>0$ and  
we again have a contradiction. 

Let $\alpha >0$ and $\beta>0$. 
If  $\varphi_0 \in C (\overline D)$ then, according to Weierstra\ss{} Theorem, 
there is a polynomial sequence $\{ P_i (x) \}$ approximating $\varphi_0 (x)$ in this space. In 
particular, for each $\varepsilon>0$, there is $i_\varepsilon \in \mathbb N$ such that 
$$
\max_{x \in \overline D} |1- \cos{(\varphi_0 (x) - P_i (x))}| < \varepsilon 
\mbox{ for all } i \geq i_\varepsilon. 
$$
Since $u  e^{\sqrt{-1} P_i (x)} \in H^+ (D)$ we see that, for all  $i \geq i_\varepsilon$,.
$$
\|Cu\|_-   \geq 
 \frac{\left|(e^{\sqrt{-1} (\varphi_0 (x)- P_i (x))} |a_0^{(2)}| u,u)_
{L^2 (D)} \right| }{\|u  e^{\sqrt{-1} P_i } \|_+ }  \geq
$$
$$
\frac{\left| (\cos{(\varphi_0 (x)- P_i (x))} |a_0^{(2)}| u,u)_{L^2 (D)} \right|  }{\|u 
 e^{\sqrt{-1} P_i } \|_+ } \geq \frac{ (1-\varepsilon) ( |a_0^{(2)}| u,u)_{L^2 (D)} }
{\|u  e^{\sqrt{-1} P_i } \|_+ }.
$$
Hence if $\varepsilon \in (0,1)$ then   
$$
\limsup_{k_j \to \infty} \frac{ ( |a_0^{(2)}| u_{k_j},u_{k_j})_{L^2 (D)}}{\|Cu_{k_j}\|_-}
\leq \frac{\limsup_{k_j \to \infty}  \|u _{k_j} e^{\sqrt{-1} P_i }  \|_+}{1-\varepsilon} 
\mbox{ for all } i \geq i_\varepsilon .
$$

On the other hand, as $|e^{\sqrt{-1} P_i }|=1 $ we conclude that 
\textcolor{red}{
\begin{equation} \label{eq.uk.weak}
\|u  e^{\sqrt{-1} P_i } \|_+^2 = \|u  \|^2_+  + 
\| ({\mathfrak D}  e^{\sqrt{-1} P_i }) u \|^2_{L^2 (D)} + 
\end{equation}
$$
2\Re{\left((({\mathfrak D}  e^{\sqrt{-1} P_i }) u , e^{\sqrt{-1} P_i } {\mathfrak D}  u)_{L^2(D)}\right)}  + 
\| \Psi  (e^{\sqrt{-1} P_i} u) \|^2_{L^2 (\partial D)} - \| \Psi  ( u) \|^2_{L^2 (\partial D)}
$$
for all $i \in \mathbb N$ and $u \in H^{+} (D)$. If  $\Psi$ is given by the multiplication on a 
function $\psi \in L^\infty (\partial D)$ then $\| \Psi  (e^{\sqrt{-1} P_i} u) \|
_{L^2 (\partial D)}  = \|\Psi  ( u) \|_{L^2 (\partial D)}$. 
If $\partial D \in C^\infty$ and $\Psi$ is a pseudodifferential operator of order $\rho$ on 
$\partial D$  then, as the multiplication on a smooth function is  a pseudodifferential operator 
of order zero, we conclude that  the commutator 
$[\Psi, e^{\sqrt{-1} P_i}] = (\Psi \circ e^{\sqrt{-1} P_i} - e^{\sqrt{-1} P_i} \circ \Psi)$ 
is a  pseudodifferential operator of order $(\rho-1)$ on $\partial D$ (see for instance 
\cite{Ho85}). By the conctruction of $\|\cdot\|_+$ and Theorem \ref{t.emb.half}, the sequence 
$\{ u_k\} $ is bounded in $H^\rho (\partial D)$ and then we can consider that 
the subsequence $\{ u_{k_j}\} $ converges weakly to zero in this space. Then 
$$
\left| \| \Psi  (e^{\sqrt{-1} P_i} u) \|_{L^2 (\partial D)} - \| \Psi  ( u) \|_
{L^2 (\partial D)} \right| =
$$
$$
\left| 
\| \Psi  (e^{\sqrt{-1} P_i} u) \|_{L^2 (\partial D)} - 
\| e^{\sqrt{-1} P_i }\Psi  ( u) \|_{L^2 (\partial D)}
\right| \leq 
\| [\Psi , e^{\sqrt{-1} P_i}] ( u) \|_{L^2 (\partial D)}
$$
for all $u \in H^+ (D)$ and hence 
\begin{equation} \label{eq.uk.weak.1}
\lim_{k_j\to \infty } \left( \| \Psi  (e^{\sqrt{-1} P_i} u_{k_j}) \|_{L^2 (\partial D)} - \| \Psi  ( u_
{k_j}) \|_{L^2 (\partial D)} \right) =0
\end{equation} 
because the operator  $[\Psi ,  e^{\sqrt{-1} P_i}]: H^\rho (\partial D) \to 
L^2 (\partial D)  $ is compact by Rellich Theorem. Thus, as $u_{k_j} \to 0$ in $L^2 (D)$ and 
$\|u_{k_j}\|_+=1$, it follows from (\ref{eq.uk.weak}) and (\ref{eq.uk.weak.1})}  that
$$
\limsup_{k_j \to \infty}  \|u _{k_j} e^{\sqrt{-1} P_i }  \|_+  =1 
\mbox{ for all } i \in \mathbb N. 
$$
Therefore, if $\beta>0$ then, by (\ref{eq.L0.PE2.inv}),  
$$
\beta =  \lim_{k_j \to \infty} \frac{- \int_{D}   \cos{(\varphi_0 +2 \varphi_\Gamma)} 
|a^{(2)}_0 (x)| |u_{k_j} (x)|^2  \ dx }{\|Cu_{k_j} \|_-}  
 \leq 
$$
$$
\limsup_{k_j \to \infty} |\theta_1| 
\frac{\int_{D}  |a^{(2)}_0 (x)| |u_{k_j} (x)|^2  \ dx }{\|Cu_{k_j} \|_-}  \leq \frac{|\theta_1|}{1-\varepsilon} \mbox{ for each } \varepsilon \in (0,1).
$$
This means that $\theta - \beta >0$ if $\theta >|\theta_1|$ and we again have a contradiction 
with (\ref{eq.contra}). Thus, (\ref{eq.ray.theta}) is fulfilled.

Finally, as $\|\delta_s L\|^2 < 1-\eta^2 (\Gamma) =1-|\theta_1|^2$ we see that
there are $\theta_2 \in (|\theta_1|, 1]$, $\gamma_0 \in [0,1)$
with $\theta_2 \sqrt{1- \gamma_0}>|\theta_1|$ 
%such that there are $\gamma_0 \in [0, 1 - |\theta_1|^2)$ 
and $\alpha_1 \in (0,\pi/2)$ such that 
$$
\|\delta_s L\| <\cos{(\alpha_1)}\Big(1-\theta_2 \Big)^{1/2}.
$$ 
Therefore, using (\ref{eq.alpha.inv}), (\ref{eq.ray.theta}) we see that 
$$
\|(L_0 + \delta_s L + \lambda ^2C) u \|_{-} \geq \sqrt{\Big(1-\theta_2^2  
\Big) \|u\|^2_+ +  \gamma_0  |\lambda |^4\|C u \|_-^2} - \|\delta_s L u \|_-  \geq 
$$
$$
 \cos{(\alpha_1)} \Big(1-\theta_2^2 \Big)^{1/2}
 \|u\|_+ + \sin{(\alpha_1)} \sqrt{\gamma_0}   |\lambda |^2\|C u \|_- - \|\delta_s L u \|_-  \geq 
$$
$$
 \Big( \cos{(\alpha_1)} \Big(1-\theta_2^2 \Big)^{1/2} -  
\|\delta_s L \| \Big) 
 \|u\|_+ + \sin{(\alpha_1)} \sqrt{\gamma_0}   |\lambda |^2\|C u \|_-  .
$$
for all $ u \in H^{+} (D)$ and  all $\lambda \in \Gamma$ with $|\lambda| \geq k_0$.
\end{proof}

We continue with the  proof of the property 1). 
For this purpose, using Lemma \ref{l.L0.PE.inv}, we conclude that the operator 
$(L_0 + \delta_s L + \lambda^2 C)$ is continuously invertible 
if $\|\delta_s L\|^2<1-\eta ^2 (\Gamma)$ and $\lambda \in \Gamma$ 
with $|\lambda| \geq k_0$. Hence we obtain 
\begin{equation}
\label{eq.spectrum.T.PE.inv}
 L (\lambda) = (\Id + \delta_c L (L_0 + \delta_s L+ \lambda^2 C)^{-1}) 
 (L_0 +\delta_s L + \lambda^2 C)
\end{equation}
for all $\lambda \in \Gamma$ with $|\lambda| \geq k_0$. 

We will show that the operator
   $\Id + \delta_c L  (L_0 + \delta_s L+ \lambda^2 C)^{-1}$
is injective for all $\lambda\in \Gamma$ such that $|\lambda|\geq k_1$ with 
some $k_1 \in {\mathbb N}$ with $k_1 \geq k_0$. Indeed, we argue by contradiction.
Suppose that for any $k \in \mathbb{N}$ there are $\lambda_k \in \Gamma$ with 
$|\lambda_k|\geq k$  and $f_k \in H^{-}   (D)$, such that
   $\| f_k \|_{-} = 1$
and
\begin{equation}
\label{eq.cee.PE.inv}
   (\Id + \delta_c L (L_0 + \delta_s L + \lambda^2_k C)^{-1}) f_k = 0.
\end{equation}
It follows from Lemma \ref{l.L0.PE.inv}  that the sequence 
$u_k := (L_0 + \delta_s L + \lambda^2_k C)^{-1} f_k$
is bounded in $H^{+}   (D)$ for all $\lambda_k \in \Gamma$ with $|\lambda_k| \geq k_0$.
Now the weak compactness principle for Hilbert spaces yields that there is a subsequence
   $\{ f_{k_j} \}$
with the property that both
   $\{ f_{k_j} \}$ and
   $\{ u_{k_j} \}$
converge weakly in the spaces $H^{-}   (D)$ and
                              $H^{+}   (D)$
to limits $f$ and $u$, respectively.
Since $\delta_c L $ is compact, it follows that the sequence $\{ \delta_c L u_{k_j} \}$
converges to $\delta_c L u$ in $H^{-}   (D)$, and so $\{ f_{k_j} \}$ converges to 
$f$ because of (\ref{eq.cee.PE.inv}).
Obviously,
$$
   \| f \|_{-} = 1.
$$
In particular, we conclude that the sequence
   $\{ \delta_c L (L_0 +  \delta_s L + \lambda^2_{k_j} C)^{-1} f_{k_j} \}$
converges to $(- f)$ whence
\begin{equation}
\label{eq.w00.PE.inv}
   f = - \delta_c L \, u.
\end{equation}

Further, on passing to the weak limit in the equality
   $f_{k_j} = (L_0 + \delta_s L + \lambda^2_{k_j} C) u_{k_j}$
we obtain
$$
   f = L_0 u +  \delta_s L u+ \lim_{k_j \to \infty} \lambda^2_{k_j}\, C\,  u_{k_j},
$$
for the continuous operator
   $L_0 +  \delta_s L : H^{+}   (D) \to H^{-}   (D)$
maps weakly convergent sequences to weakly convergent sequences.

As the operator $C$ is compact, the sequence
   $\{ C\, u_{k_j} \}$
converges to $C\,  u$ in the space $H^{-}   (D)$  and
   $C\, u \ne 0$
which is a consequence of (\ref{eq.w00.PE.inv}) and the 
injectivity of $C$ (see Lemma  \ref{l.inj.C}).
This shows readily that the weak limit
$$
   \lim_{k_j \to \infty} \lambda^2_{k_j}\, C\, u_{k_j}
 = f- L_0 u - \delta_s L u 
$$
does not exist, a contradiction.

We have proved that the operator
   $\Id + \delta_c L(L_0 +  \delta_s L +\lambda^2\, C)^{-1}$
is injective for all $\lambda\in \Gamma$ with $|\lambda|\geq k_1$. 
Since this is a Fredholm operator of index zero, it is continuously invertible.
Hence, the operators $L (\lambda )$ are continuously invertible
for all $\lambda \in \Gamma$ with sufficiently large $|\lambda|$.

Thus, 
$\{L^{-1}(\lambda)=(L_0 + \delta_c L + \delta_s L +\lambda^ 2 C )^{-1} \}
_{\lambda\in \mathbb C}$ is a meromorphic family of Fredholm operators. In particular, since 
there is a point $\gamma$ where $L(\gamma)$ is continuously invertible,   
the operators $L(\lambda)$ are continuously invertible for all $\lambda \in {\mathbb C}$ 
except a discrete countable set $\{\lambda_\nu \}$ without limit points in ${\mathbb C}$ 
(see, for instance, \cite{Keld51} or \cite{GokhKrei69}).
\end{proof}

%Comparing with the classical results  \cite{AgraVi64} on the parameter-dependent 
%elliptic operators we note that if the Shapiro-Lopatinsky conditions are violated (on the 
%smooth part of $\partial D$) then we can grant the operator $L_0$ to perturbate only.
%Thus the hypothesis of Theorem \ref{t.PE.inv} are natural.

\begin{cor} \label{c.PE.inv}
\textcolor{red}{
Let either $\Psi$ is given by the multiplication on a function $\psi \in L^\infty (\partial D)$
or $\partial D \in C^\infty$ and $\Psi$ is a pseudodifferential operator on $\partial D$.
}
Let also (\ref{eq.PE0})  hold true, $\varphi_0 \in C(\overline D) $ and 
\begin{equation} \label{eq.PE3}
\Phi =\sup_{x,y \in \overline D} (\varphi_0 (x) - \varphi_0 (y) )< 2 \pi.
\end{equation}
Under the hypothesis of Theorem \ref{t.emb.half},  
for each compact operator $ \delta_c L : H^{+} (D) \to H^{-} (D)$ and 
each bounded operator $ \delta_s L : H^{+} (D) \to H^{-} (D)$ 
with 
\begin{equation} \label{eq.PE4}
\|\delta_s L\|^2 + \left( \max{(0, -\cos{(\Phi /2)})} \right)^2 <1
\end{equation}
the operators $L (\lambda) = L_0 + \delta_s L + \delta_c L + 
\lambda^2 C$  are continuously invertible for all 
$\lambda \in {\mathbb C}$ except a  countable number of 
the characteristic values $\{\lambda_\nu\}$. 
\end{cor}

\begin{proof} As $\varphi_0 \in C(\overline D)$, the function admits maximal and minimal 
values 
$$
\Phi_1 = \min_{x\in \overline D} \varphi_0 (x), \quad 
\Phi_2 = \max_{x\in \overline D} \varphi_0 (x), 
$$
and $\Phi = \Phi_2 - \Phi_1 $. 

Then under (\ref{eq.PE0}) and (\ref{eq.PE3}) the operator 
$A(x,\partial , \lambda) =  ({\mathfrak D} \nabla)^* ({\mathfrak D} \nabla)+ 
\lambda^2 a^{(2)}_0 (x) 
$ 
satisfies conditons of Theorem \ref{t.PE.inv} on the ray 
$\Gamma_0= \{ \arg (\lambda) = -(\Phi_2 +\Phi_1) /4\}$ 
because in this case we have  
$$
-\pi <(\Phi_1 - \Phi_2)/2 \leq \varphi_0 (x) + 2 \varphi_{\Gamma_0} 
\leq (\Phi_2 - \Phi_1) /2 < \pi. 
$$ 
For this particular ray we obtain  
$$
\theta_1 = \min_{x \in \overline D}\cos{(\varphi_0 (x) + 2\varphi_{\Gamma_0} ) }
\geq \cos(\Phi/2) >-1.
$$ 
Now Theorem \ref{t.PE.inv} implies that if (\ref{eq.PE4}) is fulfilled then 
there is $\gamma_0 \in \Gamma_0$ such that the operator $L (\gamma_0)$ is continuously 
invertible. In particular, the operators $L(\lambda)$  are continuously invertible for all $
\lambda \in {\mathbb C}$ except a  countable number of 
the characteristic values $\{\lambda_\nu\}$. 
\end{proof}

\section{On the completeness of root functions}
\label{s.roots}

We are interested in studying  the completeness of root functions related to the  mixed 
problem in Sobolev type spaces $H^+ (D)$, $H^{-} (D)$.

To this purpose we recall some basic definitions. Suppose $\lambda_{0} \in 
\mathbb{C}$ and $F (\lambda)$ is a holomorphic function in a punctured 
neighbourhood of $\lambda_{0}$ which takes on its values in the space $\mathcal{L} 
(H_1,H_2)$ of bounded linear operators acting from a Hilbert space $H_1$ to a Hilbert 
space $H_2$. The point $\lambda_{0}$ is called a characteristic point of $F (\lambda)$ if
there exists a holomorphic function $u (\lambda)$ in a neighborhood of
$\lambda_{0}$ with values in $H_1$, such that $u (\lambda_{0}) \neq 0$ but
$F (\lambda) u (\lambda)$ extends to a holomorphic function (with values in $H_2$) near the 
point $\lambda_{0}$ and vanishes at this point. Following \cite{GokhSiga1}, 
we call $u (\lambda)$ a root function of the family $F (\lambda)$ at $\lambda_{0}$.

If $N$ is the order of zero of the holomorphic function $F (\lambda) u (\lambda)$ 
at the point $\lambda_0$ then we have 
\begin{equation} \label{eq.roots}
\sum_{j=0}^{m} F _{m-j}u_j = 0 \mbox{ for all } m\in {\mathbb Z}_+ \mbox{ with } 0 \leq
 m \leq N-1 
\end{equation}
where $u_j =\frac{1}{j!}\frac{d ^j u}{d z^j} (\lambda_0) \in H_1$ and $F_j = 
\frac{1}{j!} \frac{d ^j F}{d z^j} (\lambda_0) \in \mathcal{L} (H_1,H_2)$,  $j \in \mathbb N$.
The vector $u_0$ is called an eigenvector of the family $F(\lambda)$ at the point 
$\lambda_0$ and the vectors $u_j$, $1\leq j \leq N-1$, are said to be associated vectors for 
the eigenvector $u_0$. If the linear span of the set of all eigen- and associated 
vectors  the family $F(\lambda)$ is dense in $H_1$ one says that the root 
functions of the family $F (\lambda)$ are complete in $H_1$. 

However, the notion of root function of a holomorphic family is a generalization of the notion 
of a root vector of a linear operator. Namely, recall that a complex number $\mu \in {\mathbb
 C}$ is said to be an eigenvalue of a linear operator $T:H \to H$ in a Hilbert space $H$ if 
there is a non-zero element $u \in D_{T}$, such that $(T - \mu \Id)  u = 0$, where $\Id$ is 
the identity operator in $H$. The element $u$ is called an eigenvector of $T$ corresponding 
to the eigenvalue $\mu$. A non-selfadjoint compact operator might have no eigenvalues.
However, each non-zero eigenvalue (if exists) is of finite multiplicity, see for instance
 \cite{DunfSchw63}. Similarly to the Jordan normal form of a linear operator on a 
finite-dimensional vector space one uses the more general concept of root vectors of operators.

More precisely, a non-zero element $u \in H$ is called a root vector of $T$ corresponding to an
eigenvalue $\mu_0 \in {\mathbb C}$ if $u \in D_{(T-\mu_0 \Id)^k}$, for all 
$1\leq k \leq m$ and $(T - \mu_0\Id)^m u = 0$ for some natural number $m$.
The set of all root vectors corresponding to an eigenvalue $\mu_0$  (complemented by zero 
element) forms a vector subspace in $H$ whose dimension is called the (algebraic)
multiplicity  of $\mu_0$.

Note that under (\ref{eq.PE0}) the multiplication on the function $a^{(2)}_0 \in L^\infty (D)$ 
induces a bounded injective  operator in the space $ L^2 (D) $; it is continuously 
invertible under (\ref{eq.PE1}). We will denote this operator 
by $C_0 $. Then we can factorize $C = \iota ^\prime C_0 \iota$.

\begin{lem} \label{l.reduce}
If (\ref{eq.PE0}) is fulfilled then, for the holomorhic Fredholm family  
$
L(\lambda) = L (0) + \lambda^2 C: H^{+}   (D) \to H^{-}   (D) 
$  
the set of all its root functions coincides with the set of all the root 
vectors of one  of the following closed densely defined linear operators: 
$$
C^{-1} L (\gamma): H^+ (D) \to H^{+} (D) \mbox{ and } L (\gamma) 
C^{-1} : H^- (D) \to H^{-} (D), 
$$
where $ \gamma \in \mathbb C$ is an arbitrary point.  
Besides, if there is a point $\gamma _0 \in \mathbb C$ where the operator  
$L (\gamma_0) =  L (0) + \lambda^2_0 C$ is continuously 
invertible, it also coincides with the set of all the root 
vectors of one  of the following bounded linear operators: 
$$
L^{-1}  (\gamma_0) C: H^+ (D) \to H^{+} (D) \quad C L^{-1}  (\gamma_0) : 
H^- (D) \to H^{-} (D), 
$$
$$ 
\iota L^{-1}  (\gamma_0)  \iota ^\prime C_0 : L^2 (D) \to L^{2} (D).  
$$
\end{lem}

\begin{proof} Follows immediately from (\ref{eq.roots}).
\end{proof}

To formulate the completeness results regarding to parameter-dependent elliptic 
operators we need the notion of a compact operator of finite order. If $T: H \to H$ is 
compact, then the operator $T^* T$  is compact, selfadjoint and non-negative. Hence it 
follows that $T^\ast T$  possesses a unique non-negative selfadjoint compact square root  
$(T^*T)^{1/2}$ often denoted by $|T|$. By the Hilbert-Schmidt Theorem  the operator 
$|T|$ has countable system of non -negative eigenvalues $s_\nu (T)$
which are called the $s\,$-numbers of $T$.
It is clear that if $T$ is selfadjoint then $s_\nu = |\mu_\nu|$, where $\{ 
\mu_\nu \}$ is the system of eigenvalues of $T$.
The operator $T$ is said to belong to the Schatten class ${\mathfrak S}_p$, 
with $0 < p < \infty$,
if
$$
   \sum_\nu |s_\nu (T)|^p < \infty.
$$
After M.V. Keldysh a compact operator $T$ is said to be of finite order if it 
belongs to a Schatten class ${\mathfrak S}_p$. The infinum $\mbox{ord} \ (T)$ 
of such numbers $p$ is called the order of $T$.

Let us denote by ${\mathfrak C} : H^{+} (D)\to H^{-} (D)$ the linear bounded operator 
induced by the term $(|a_0^{(2)}| u, v)_{L^2 (D)}$. Note that under (\ref{eq.PE0}) 
the multiplication on the function $|a^{(2)}_0| \in L^\infty (D)$ induces a bounded 
injective selfadjont operator ${\mathfrak C}_0 : L^2 (D) \to L^2 (D)$; it is 
continuously invertible under (\ref{eq.PE1}). 

In the following theorem $h (\cdot,\cdot)$ stands for  the Hermitian form 
$$
h (u,v) = (|a^{(2)}_0 | u,v)_{L^2 (D)}.
$$
We note that, under (\ref{eq.PE0}), it defines a scalar product on $L^2 (D)$; 
this Hilbert space we denote by $L^2_h  (D)$.  The corresponding norm is not stronger than 
$\|´\cdot\|_{L^2 (D)}$, it is equivalent to the original norm of this space if (\ref{eq.PE1}) 
is fulfilled. 

\begin{thm} \label{t.PE.self.complete}
Let (\ref{eq.PE0})  hold true. Under the hypothesis of Theorem \ref{t.emb.half},  
the operators 
$$
L^{-1}_0 {\mathfrak C}: H^{+} (D)\to H^{+} (D), \,  
{\mathfrak C} L^{-1}_0 : H^{-} (D)\to H^{-} (D), \, 
 \iota L_0^{-1}  \iota' {\mathfrak C}_0 :  L^{2} (D)\to L^{2} (D)
$$ 
are compact and  their orders are finite:
$$
\mbox{ord} \, ({\mathfrak C} L^{-1}_0) = \mbox{ord} \, (L^{-1}_0 {\mathfrak C}) = 
\mbox{ord} \, (\iota L_0^{-1}  \iota' {\mathfrak C}_0 )= n/(2\rho+1). 
$$
Moreover, the operators $L^{-1}_0 {\mathfrak C}$ and 
$  {\mathfrak C} L^{-1}_0 $ are selfadjoint. 
Besides, the operators have the same systems of eigenvalues $\{ \mu_\nu \}$,  the system  
$\{ b^{(+)}_\nu\}$ of eigenvectors of the operator $L^{-1}_0 {\mathfrak C} $ is complete in 
the spaces  $H^{+} (D)$,  $L^{2} (D)$ and $H^{-} (D)$. Moreover,  
the system $\{ b^{(+)}_\nu\}$ is an orthonormal basis in $H^{+} (D)$, the system 
$\{  b^{(-)} _\nu = {\mathfrak C}  b^{(+)}_\nu \}$ of eigenvectors of the operator 
${\mathfrak C} L^{-1}_0 $ is an orthogonal  basis in $H^{-} (D)$, the
system $\{  b^{(0)} _\nu = \iota  b^{(+)}_\nu \}$  
of eigenvectors of the operator $\iota L_0^{-1}  \iota' {\mathfrak C}_0  $ is  an 
orthogonal basis  in the space $L^2_h (D)$ and the system 
$\{\sqrt{|a^{(2)}_0|} b^{(+)}_\nu \}$ is an orthogonal basis in $L^2 (D)$.
If, in addition, (\ref{eq.PE1}) holds then the operator 
$\iota L_0^{-1}  \iota' {\mathfrak C}_0  $ is selfadjoint in $L^2 _h (D)$.
\end{thm}

\begin{proof} First of all we note that 
$$
L_0^{-1} {\mathfrak C} =  L_0^{-1}  \iota' {\mathfrak C}_0 \iota : H^{+} (D)\to H^{+} 
(D), \quad {\mathfrak C} L_0^{-1}  =  \iota' {\mathfrak C}_0 \iota L_0^{-1} :  
H^{-} (D) \to H^{-} (D) .
$$ 
Under the hypothesis of Theorem \ref{t.emb.half},  $H^+ (D)$ is continuously embed\-ded 
to $H^{s} (D)$ and then, according to Rellich Theorem, the embedding  
$\iota : H^+ (D) \to L^2 (D)$ is compact. Hence the operators $L_0^{-1} {\mathfrak C}$, ${
\mathfrak C} L_0^{-1}$ and $ \iota L_0^{-1}  \iota' {\mathfrak C}_0 $ are compact. 

Moreover, 
\begin{equation} \label{eq.complete.L2}
(L_0^{-1} {\mathfrak C}  u,v)_+ = < \iota' {\mathfrak C}_0  \iota u ,v> = 
( {\mathfrak C}_0  \iota u  ,\iota v) _{L^2 (D)} = 
\int_D |a_0^{(2)} (x)| u(x) \overline v(x) \ dx ,  
\end{equation}
$$
(u, L_0^{-1} {\mathfrak C}  v)_+ = \overline{(L_0^{-1} {\mathfrak C}_0  v, u)_+} =
 \overline{ ( {\mathfrak C}_0  \iota v , \iota u) _{L^2 (D)}} =  
\int_D |a_0^{(2)} (x)| u(x) \overline v(x) \ dx
$$
for all $u,v \in H^+ (D)$, i.e.  the operator  $L_0^{-1} {\mathfrak C}$ is selfadjoint. Then  
$$
(L_0^{-1} {\mathfrak C}  u,u)_+ = 
\int_D |a_0^{(2)} (x)| |u(x)|^2 \ dx \geq 0 
$$
for all $u\in H^+ (D)$, see (\ref{eq.complete.L2}). Hence the operator 
$L_0^{-1} {\mathfrak C}$ is non-negative 
and then it is positive because both   $L_0^{-1}$ and ${\mathfrak C}$ are injective.

According to \cite[Corollary 3.5]{ShlTark12}, the operator 
$\iota L_0^{-1}  \iota':L^2 (D) \to L^2 (D) $ is compact selfadjoint and its order is finite: 
$$
\mbox{ord} \, (\iota L_0^{-1}  \iota') = n/ (2\rho+1).
$$ 
As ${\mathfrak C}_0:L^2 (D) \to L^2 (D) $ is bounded, the operators 
$\iota L_0^{-1}  \iota'$ and $ \iota L_0^{-1}  \iota' {\mathfrak C}_0$ 
have the same orders  (see \cite[Ch.~2, \S~2]{GokhKrei69}, 
\cite{DunfSchw63} or elsewhere).  

As the operator $\iota$ is injective, we see that 
$$
( L_0^{-1}  \iota' {\mathfrak C} _0 \iota  - \mu \Id ) u =0
$$
if and only if 
$$
( \iota L_0^{-1}  \iota' {\mathfrak C}_0   - \mu \Id ) \iota u =0. 
$$
Therefore 
$$
( L_0^{-1}  \iota' {\mathfrak C}_0  \iota - \mu \Id )^m  u =0
$$
with some $m \in \mathbb N$ if and only if 
$$
( \iota L_0^{-1}  \iota' {\mathfrak C} _0   - \mu \Id )^m \iota u =0. 
$$
Thus, the sets of eigenvalues and root vectors  of the operator $ L_0^{-1} {\mathfrak C}$ 
coincides with  the sets of eigenvalu\-es and root vectors of  the operator $ \iota L_0^{-1}  
\iota' {\mathfrak C} _0 $.  Besides the multiplicities of the eigenvalues  coincide, too. 
Hence the orders of the operators $ \iota L_0^{-1}  \iota' {\mathfrak C}_0$ and 
$L_0^{-1} {\mathfrak C}$ coincide. 

By the Hilbert-Schmidt Theorem, there is an orthonormal basis 
$\{b^{(+)}_\nu\}$ in $H^+ (D)$,   consisting  of the 
eigenvectors corresponding to the eigenva\-lues $\{\mu_\nu \}$  of the operator  
$L_0^{-1} {\mathfrak C}$. Hence, by the discussion above, the vectors $\iota b^{(+)}_\nu = 
b^{(0)}_\nu \in L^2 (D)$, $\nu \in \mathbb N$ are the eigenvectors 
of the operator $ \iota L_0^{-1}  \iota' {\mathfrak C}_0$ corresponding to the 
eigenva\-lues $\{\mu_\nu \}$ of the operator  $L_0^{-1} {\mathfrak C}$.
As $H^{+} (D)$ is dense in $L^2 (D)$, the system $\{\iota b^{(+)}_\nu  \}$ is complete 
in $L^2 (D)$. 

On the other hand, 
$$
h (\iota u, \iota v) = 
 ( |a_0^{(2)} (x)| \iota u, \iota v)_{L^2 (D)}  =  
<{\mathfrak C} u, v> = (L_0^{-1}{\mathfrak C} u, v)_+
$$
for all $u,v \in H^{+} (D)$. In particular, the system
$\{\iota b^{(+)}_\nu  =  b^{(0)}_\nu \}$ is 
orthogonal in $L_h^2 (D)$. It is complete in $L_h^2 (D)$ because 
the space can be considered as the completion of $L^2 (D)$ with 
respect to $h (\cdot, \cdot )$. In particular, the system 
$\{\sqrt{|a^{(2)}_0|} b^{(+)}_\nu \}$ is orthogonal  in $L^2 (D)$. 
If a vector $u$ from $L^2 (D)$ is orthogonal to the system 
$\{\sqrt{|a^{(2)}_0|} b^{(+)}_\nu \}$ in  $L^2 (D)$ then the  
 vector $ \sqrt{|a^{(2)}_0|} u \in L^2 (D)$ is orthogonal to the system 
$\{ \iota b^{(+)}_\nu \}$ in  $L^2 (D)$. Since $\{ \iota b^{(+)}_\nu \}$ is complete 
in $L^2 (D)$ we see that $ \sqrt{|a^{(2)}_0|} u =0$ almost 
everywhere in $D$ and then $u=0$ because of (\ref{eq.PE0}). Hence 
 the system $\{\sqrt{|a^{(2)}_0|} b^{(+)}_\nu \}$ is an orthogonal basis  in $L^2 (D)$. 

Now, by the very construction,  the  space $H^+ (D)$ 
is dense in $H^- (D)$ and hence the system  $\{b^{(+)}_\nu\}$ is complete $H^- (D)$. 
Moreover, $$
(a^{(2)}_0 b^{(+)} _\nu, a^{(2)}_0 b^{(+)} _k)_- = 
( L^{-1}_0 {\mathfrak C}  b^{(+)}_\nu, L^{-1}_0 {\mathfrak C}  
b^{(+)}_k)_+ = \mu_\nu \mu_k \delta_{\nu,k} 
$$
i.e. $ \{ {\mathfrak C} b^{(+)}_\nu\}$  is orthogonal  in $H^- (D)$.
It is complete because 
$$
L_0^{-1} u = \sum_\nu (L_0^{-1} u , b^{(+)}_\nu)_+ b^{(+)}_\nu  
$$ 
for each $u \in H^{-} (D)$ by the discusion above and then 
$$
u = \sum_\nu (L_0^{-1} u , b^{(+)}_\nu)_+ L_0 b^{(+)}_\nu  = 
\sum_\nu \frac{(L_0^{-1} u , b^{(+)}_\nu)_+}{\mu_\nu} C b^{(+)}_\nu .
$$ 
Since the operator $L^{-1}_0 {\mathfrak C}: H^{+} (D) \to H^{+} (D)$ is selfadjoint, 
we have 
$$
( {\mathfrak C} L^{-1}_0 u ,v)_- = (L^{-1}_0 {\mathfrak C} (L^{-1}_0 u ),L_0 ^{-1} v)_+ 
= (  (L^{-1}_0 u ), L^{-1}_0 {\mathfrak C} L_0 ^{-1} v)_+ = (u , {\mathfrak C} L^{-1}_
0  v)_-
$$
for all $u,v \in H^{-} (D)$, i.e. the operator ${\mathfrak C} L^{-1}_0 : H^{-} (D) 
\to H^{-} (D)$ is selfadjoint, . Hence, by the Hilbert-Schmidt Theorem, there is an 
ortho\-nor\-mal basis $\{  b^{(-)}_\nu\}$ in $H^- (D)$, consisting  of the eigenvectors of 
the operator ${\mathfrak C} L_0^{-1}$. 

On the other hand, as for all $u \in H^+ (D)$ the identity 
\begin{equation} \label{eq.hat.b}
{\mathfrak C} ( L^{-1}_0 {\mathfrak C}  -  \mu \Id) u =    
( {\mathfrak C} L^{-1}_0   -  \mu \Id) {\mathfrak C} u 
\end{equation}
holds true and the operator ${\mathfrak C} $ is injective, we conclude that the systems of the 
eigenvalues of the operators $L^{-1}_0 {\mathfrak C}  $ and ${\mathfrak C}  L^{-1}_0 $ 
coincide. Moreover the eigenvalues has the same multiplicities and then  $\mbox{ord} \, 
({\mathfrak C}  L^{-1}_0) = \mbox{ord} \, (L^{-1}_0 {\mathfrak C} )$. Therefore  
(\ref{eq.hat.b}) implies that for each $\nu \in \mathbb N$ the vector  
$  b^{(-)}_\nu =  {\mathfrak C}  b ^{(+)}_\nu$ is an 
eigenvectors of the operator ${\mathfrak C}  L^{-1}_0$ corresponding to the eigenvalue 
$\mu_\nu$, too. 
In particular, we can consider the system $ \{ {\mathfrak C}  b^{(+)} _\nu \}$  as 
an orthogonal basis  in $H^- (D)$, consisting of the eigenvectors of the 
operator ${\mathfrak C}  L^{-1}_0$. 

Finally, if (\ref{eq.PE1}) is fulfilled then $L^2_h (D)$ is a Hilbert space 
coinciding with $L^2 (D)$ as the linear space and having an equivalent norm.
Then  
$$
h (\iota L_0^{-1}  \iota' {\mathfrak C}_0 u,v) = 
 (\iota L_0^{-1}  \iota' {\mathfrak C}_0 u, |a_0^{(2)} (x)| v)_{L^2 (D)}  =  
$$
$$
( |a_0^{(2)} (x)|  u, \iota L_0^{-1}  \iota'  {\mathfrak C}_0 v)_{L^2 (D)} 
= h(  u, \iota L_0^{-1}  \iota'  {\mathfrak C}_0 v)
$$
for all $u,v \in L^2 (D)$, i.e. the operator $\iota L_0^{-1}  \iota' {\mathfrak C}_0 :  
L^{2}_h (D)\to L^{2}_h (D)$ is selfadjoint. 
\end{proof}

Now we can use the famous Keldysh' Theorem on the weak perturbation 
of compact selfadjoint operators (see, \cite{Keld51} or \cite{GokhKrei69}). 

\begin{cor} \label{c.PE.self.complete} 
Let (\ref{eq.PE0})  hold true. Under the hypothesis of Theorem \ref{t.emb.half},  
for each compact operator $ \delta_c L : H^{+} (D) \to H^{-} (D)$ we have 

1)   for any $\varepsilon > 0$ all the characteristic values $\lambda_\nu$ (except for a finite 
number) of the family $L (\lambda) = L_0 + \delta_c L + \lambda^2 \mathfrak C$  belong 
to the corners  
\begin{equation} \label{eq.corners}
M_\varepsilon=\{ | \arg( \lambda)  - \pi/2|<\varepsilon \} ,\quad 
M_{-\varepsilon} =\{ | \arg( \lambda)  + \pi/2|<\varepsilon \}
\end{equation}
and $\lim_{\nu \to \infty} |\lambda_\nu|=+\infty$;

2)  the system of   root vectors  of the family 
$L (\lambda) = L_0 + \delta_c L + \lambda ^2 \mathfrak C$ is   
complete in the spaces $H^{+} (D)$, $L^2 (D)$ and  $H^{-} (D)$. 
\end{cor}

\begin{proof} Under  the hypothesis of the theorem, $\Phi =0$ and Corollary 
\ref{c.PE.inv} implies that  there is $\gamma_0 $  on the ray $\Gamma_0 = 
\{\arg{(\lambda)} =0 \}$ such that $L (\gamma_0)$ is continuously invertible.
Thus, according to Lemma \ref{l.reduce} the proof of the statements 1) and 2) of the 
theorem can be reduced to the investigation of the properties of the compact operator 
$L^{-1} (\gamma_0) \mathfrak  C$. 

On the other hand, the operator 
$$
\delta L = (L (\gamma_0) - L_0) = \delta_c L + \gamma_0^2 {\mathfrak C}
$$ 
is compact and then the operator $ L^{-1} (\gamma_0) (\delta L)   $ is compact, too.
Easily, we obtain 
$$
L^{-1} (\gamma_0) - L^{-1}_0   = -  L^{-1} (\gamma_0)  (\delta L) L^{-1}_0  . 
$$
Hence the operator 
$$
L^{-1} (\gamma_0) {\mathfrak  C} =  L^{-1}_0 {\mathfrak  C} - 
 L^{-1} (\gamma_0)   (\delta L) (L^{-1}_0 {\mathfrak  C}) 
$$
can be considered as a weak perturbation of the selfadjoint operator 
$L^{-1}_0 {\mathfrak  C}$ (see \cite{Keld51} or \cite{GokhKrei69}). 
Indeed, according to Theorem \ref{t.PE.self.complete} the order of the operator 
$L^{-1}_0 {\mathfrak  C} $ is finite. As the operators  
$L^{-1}_0 {\mathfrak  C} $ and $ L^{-1} (\gamma_0) {\mathfrak  C}$  are  injective,  
the statements the completeness of the root vectors $\{ b_\nu\}$ of the operator 
$ L^{-1} (\gamma_0) {\mathfrak  C}$ in the space $H^{+} (D)$ follows from famous 
Keldysh' Theorem (see \cite{Keld51} or \cite{GokhKrei69}). Moreover, 
this theorem also iplies that  the sequence 
$\{ \mu_\nu\} $ of its eigenvalues converges to zero and belongs to the corner 
$\{ | \arg( \mu) |<\varepsilon \} $ (except for a finite number of its elements).
Then, by the construction of the spaces, the system $\{  b_\nu\}$ is complete in 
$L^{2} (D)$ and $H^{-} (D)$, too. 

As, $\mu_\nu = (\gamma_0^2 - \lambda_\nu^2)^{-1}, \, \nu \in {\mathbb N} $, the 
property 1) of the theorem holds, too.
%$\lim_{\nu \to \infty}|\lambda_\nu |=+ \infty$ and 
%for each $\varepsilon >0$ there is $N_\varepsilon$ such that
%$$
%\left|\arg{\left( -\overline \lambda_\nu^2 \left(1-\frac{\gamma_0^2}{\overline 
%\lambda_\nu^2} \right)\right) } \right|< 2\varepsilon
%\mbox{ for all } \nu \geq N_\varepsilon.
%$$
%This is the same as to say that for all $\nu \geq N_\varepsilon$
%one of the following inequalities holds
%$$
%\left|\arg{(\lambda_\nu)} + \frac{\pi}{2} -   
%1/2\arg{\left(1-\frac{\gamma_0^2}{\overline \lambda_\nu^2}\right)}   \right|
%< \varepsilon , \quad \left|\arg{(\lambda_\nu)} - \frac{\pi}{2} -   
%1/2\arg{\left(1-\frac{\gamma_0^2}{\overline \lambda_\nu^2}\right)}   \right|
%< \varepsilon .
%$$
%Finally, since $|\lambda_\nu | \to + \infty$ as $\nu \to +\infty$ we may consider that 
%for any $\varepsilon >0$ all the points $\left(1-
%\frac{\gamma_0^2}{\overline \lambda_\nu^2}\right)$ 
%(except for a finite number) belong to the corners $M_{\pm\varepsilon/2}$ 
%and therefore for each $\varepsilon >0$ there is $\tilde N_\varepsilon$
%such that all the eigenvalues $\{\lambda_\nu \}_{  \nu \geq \tilde N_\varepsilon} $ 
%belong to the corners (\ref{eq.corners}).
\end{proof}

Finally, we may apply the method of rays of minimal growth of the resolvent to obtain 
the completeness of root vectors in the case of more general perturbations.

\begin{thm} \label{t.PE.complete} 
\textcolor{red}{
Let either $\Psi$ is given by the multiplication on a function $\psi \in L^\infty (\partial D)$
or $\partial D \in C^\infty$ and $\Psi$ is a pseudodifferential operator on $\partial D$.
}
Under the hypothesis of Theorem \ref{t.emb.half},
let also (\ref{eq.PE0}) and \begin{equation} \label{eq.PE5}
\Phi =\sup_{x,y \in \overline D} (\varphi_0 (x) - \varphi_0 (y) )< \pi (2\rho+1) /2n .
\end{equation} 
hold true. If 
$\varphi_0 \in C^{0,1} (\overline D)$ and 
\begin{equation} \label{eq.PE6} 
\|\delta_s L\|^2 + \left( 
\max{(0, - \cos{((\pi (2\rho+1) -2n\Phi) /4n)})} \right)^2 < 1
\end{equation} 
 then we have  

1) for any $\varepsilon > 0$ 
all the characteristic values $\lambda_\nu$ (except for a finite number) of the family 
$L (\lambda) = L_0 + \delta_s L + \delta_c L +\lambda^2  C$  belong to the corners  
$$
\{ | \arg( \lambda)  \pm \pi/2|<\pi (2\rho+1)/2 n + \varepsilon \} 
$$ 
and   $\lim_{\nu \to \infty} |\lambda_\nu|=+\infty$;

2)  the system of   root vectors  of the family 
$L (\lambda) = L_0 + \delta_c L + \delta_s L + \lambda ^2  C$ is  
complete in the spaces $H^{+} (D)$, $H^{-} (D)$ and $L^2 (D)$. 
\end{thm}

\begin{proof} First of all, we note (\ref{eq.PE5}) implies (\ref{eq.PE3})
and then Corollary \ref{c.PE.inv} yields the existence of a number 
$\gamma_0 \in \mathbb C$ such that $L (\gamma)$ is continuously invertible. 
In particular, the operators $L(\lambda)$ are continuously invertible for all $\lambda \in {\mathbb C}$ except a  
countable number of the characteristic values $\{\lambda_\nu\}$. 

As the operator 
$\gamma^2_0 C: H^{+} (D)  \to H^{-} (D)$   is compact, the family 
$$
\tilde L (\tilde \lambda) = L_0 + \delta_s L + \tilde \delta _c L + \tilde \lambda^2 C 
$$
with $\tilde \delta _c L = \delta _c L + \gamma^2_0 C$ and $\tilde \lambda^2 = \lambda^2 - \gamma_0^2$  
satisfies conditions of Theorem \ref{t.PE.complete}, too. 
Moreover, the operator $\tilde L(0) = L (\gamma_0)$ is continuously 
invertible. Since the root functions and root vectors of the families 
$\tilde L (\tilde \lambda)$  and $L (\lambda)$ have obvious relations, 
we can replace the family $L (\lambda)$ by the family $\tilde L (\tilde \lambda)$.
Thus without loss of generality we may consider that the operator $L (0)$ is continuously 
invertible.

As $0<(2\rho+1)/2 n \leq 1/2$, it follows from (\ref{eq.PE5}) that 
there is $0<\epsilon< \pi/2$ such that 
\begin{equation} \label{eq.PE7}
 \pi (2\rho+1) /2n -\Phi  = 2\epsilon.
\end{equation}
Then under (\ref{eq.PE0}) and (\ref{eq.PE5}) the operator $A(x,\partial , \lambda) =  
({\mathfrak D} \nabla)^* ({\mathfrak D} \nabla)+ \lambda^2 a^{(2)}_0 (x) $ 
satisfies conditions of Theorem \ref{t.PE.inv} on any ray $\Gamma$ with 
\begin{equation} \label{eq.ell.rays}
- (\pi+ \Phi_1 -a\epsilon)/2 <\varphi_\Gamma < (\pi-\Phi_2 - a\epsilon)/2  
\end{equation}
where $0<a<1$ is an arbitrary number. Indeed, in this case (\ref{eq.PE5}) implies that the 
interval $(-(\pi+\Phi_1-a\epsilon), (\pi-\Phi_2-a\epsilon)) $ is not empty and 
$$
\varphi_0 (x) + 2\varphi_\Gamma 
\leq  \Phi_2 + \pi- \Phi_2  -a\epsilon\leq \pi-a\epsilon <\pi, 
$$
$$
\varphi_0 (x)   + 2\varphi_\Gamma
\geq  \Phi_1  - \Phi_1 -\pi +a\epsilon \geq -\pi + a\epsilon > - \pi. 
$$ 
For these  rays we have 
$$
\theta_1 (\Gamma) = \min_{x \in \overline D}\cos{(\varphi_0 (x) + 2\varphi_{\Gamma} ) }
\geq \cos(\pi - a \, \epsilon) = - \cos{a \, \epsilon}>-1.
$$
Hence (\ref{eq.PE7}) implies  that there is a number $a \in (0,1)$ 
such that we have
$$
\|\delta_s L\|^2 + \left(\max{\left(0, - \cos{\left(\frac{a}{2} \, \frac{\pi (2\rho+1) -2n\Phi} 
{4n}\right)}\right)} \right)^2 < 1.
$$

Thus, according to Lemma \ref{l.reduce} the proof of the statements 1) and 2) of the 
theorem can be reduced to the investigation of the properties one of the operators  
$L^{-1} (0) C$ and $L (0) C^{-1}$ . 

If $\varphi_0 \in C^{0,1} (\overline D)$ then the multiplication 
on the function $e^{\sqrt{-1} \varphi_0} \in C^{0,1} (\overline D)$ induces 
a bounded linear operator $\delta_C : H^+ (D) \to H^{+} (D)$. 
Hence the operator $C L^{-1} (0) $ can be presented in 
the following form:
$$
C L ^{-1} (0)  =  ({\mathfrak C} L_0^{-1}) \, L_0 \, \delta_C \, L ^{-1}(0)  .
$$
It follows from Theorem \ref{t.PE.self.complete} that the operator 
${\mathfrak C} L_0^{-1} $ belongs to the Schatten class 
${\mathfrak S}_{n/ (2\rho+1) + \varepsilon}  $ with any $\varepsilon >0$.
As compositions with bounded operators preserve the Schatten class 
(see \cite[Ch. 2, \S 2]{GokhKrei69}) we conclude that the compact 
operator $C L ^{-1} (0) $ belongs to the Schatten class 
${\mathfrak S}_{n/ (2\rho+1) + \varepsilon}  $ with any $\varepsilon >0$, too.

Now, the estimate (\ref{eq.rays.minimal}) implies that on the rays 
$$ 
- (\pi + \Phi_1 - a\epsilon)< \arg{(\lambda^2)} < \pi - \Phi_2 - a\epsilon
$$
we have 
$$
\| (L (0)  - (- \lambda^2)C ) u \|_- \geq c _\Gamma
|\lambda|^2 \|Cu\|_- \mbox{ for all } u \in H^{+} (D).
$$ 
with a constant $c_\Gamma>0$ independent on $u$ for all sufficiently large $\lambda^2$ 
on each of the rays.  Hence the rays 
$$ 
a \epsilon- \Phi_1< \arg{(\mu)} < 2\pi - \Phi_2 - a\epsilon
$$
are the rays of the minimal growth of the resolvent of the closed operator 
$L (0) C^{-1}$, i.e.
\begin{equation}
\label{eq.rays.minimal.mu}
\| (L (0) C^{-1} - \mu)^{-1} w \|_- \leq c^{-1}_\Gamma |\mu|^{-1} \|w\|_- \mbox{ for all } 
w \in H^{-} (D)
%\mbox{Dom} (L (0) C^{-1} - \mu)^{-1})
\end{equation}
for all sufficiently large $\mu$ on each of the rays,
see, for instance, \cite{Agmo62}. Moreover, it follows from (\ref{eq.PE7}) and 
(\ref{eq.rays.minimal.mu}) the angle between any two
neighbouring rays of minimal growth are less than $\pi(2\rho+1)/2n$ if $0<a<1$. 

Thus the statement of the theorem follows 
from the standard arguments with the use of Phragmen-Lindel\"{o}f theorem
which go back at least as far as \cite{Agmo62}.
\end{proof}

\begin{rem} \label{eq.multi.compl} \textit{Actually, it follows from the reducing procedure of Lemma \ref{l.reduce} that in Corollary \ref{c.PE.self.complete}  and Theorem 
\ref{t.PE.self.complete} we should claim the multiple  (double) completeness instead of 
the completeness (see \cite{Keld51}, \cite{Yak90} and elsewhere). }
\end{rem}

\section{An example}
\label{s.exmp}

Consider an instructive example. 

%\begin{exmp} \label{ex.NS.1} 
Let $n=2$ and $A^{(2)}_0$ be a $(2\times 2)$ matrix with real-valued entries of 
class $L^\infty (D)$. Consider the Lam\'e type system
\begin{equation*} 
%\label{eq.Lame.1}
\tilde A (x,\partial,\lambda)  V (x) = -  \vartheta \Delta_2 I_2 V (x)  
- (\vartheta + \vartheta_1) \nabla_2 \mbox{div}_2 V (x) 
+ \lambda^2 A^{(2)}_0 (x)  V (x)
\end{equation*}
where $V (x)= (V_1 (x),V_2 (x))$ is an unknown vector, 
$I_2$ is the identity $(2\times 2)$-matrix, $\Delta_2$ the Laplace operator, 
$\nabla_2 $  and $\mbox{div}_2 $ 
are the gradient operator and the divergence operators in ${\mathbb R}^2$ respectively 
and   $\vartheta $, $\vartheta_1 $ are  the Lam\'e parameters.  
This operator plays an essential role in the two-dimensional Linear Elasticity Theory 
(see, for instance, \cite{Fi72}); the vector function 
$V (x)$ represents the discplacement of points of an 
elastic body. This operator can also be considered as a part of linearisation system of the 
stationary version of the two-dimensional Navier-Stokes type equations 
for viscous compressible fluid  with known pressure and unknown velocity vector  
$V (x)$ (see \cite[\S 15]{LaLi59}); in this case the Lam\'e parameters  represent 
viscosities. As it is know, the system is strongly elliptic and formally selfadjoint 
non-negative  if $\vartheta>0$, $2\vartheta + \vartheta_1 >0$.

Let us consider  a very special case where  the fisrt Lam\'e parameter $\vartheta_1$ 
is negative and $\vartheta_1 = - \vartheta$. Then $\tilde A (x,\partial,\lambda)$ reduces to 
\begin{equation} 
\label{eq.Lame.1}
\tilde A (x,\partial,\lambda)   = -  \vartheta \Delta_2 I_2 
+ \lambda^2 A^{(2)}_0 (x)  .
\end{equation}
On the other hand,
$$
- \Delta_2 I_2 V   = \mbox{rot}_2^* \mbox{rot}_2 V +  \mbox{div}_2^* \mbox{div}_2 V  
$$
where $\mbox{rot}_2 V = (\partial_1 V_2 -  \partial_2 V_1) $ is  the rotation  operator in ${
\mathbb R}^2$ and $\mbox{rot}_2^*$,  $\mbox{div}_2^*$ are the formal 
adjoint operators for $\mbox{rot}_2$,  $\mbox{div}_2$ respectively.
 
Assume now that the matrix $A^{(2)}_0 (x)$ has the following form 
$$
A^{(2)}_0 (x) = \alpha (x) U (x)
$$
where $\alpha (x) \in L^\infty (D)$ is a non-negative function  and 
$$
U (x) = \left( \begin{array}{ll} U_1 (x) & - U_2(x) \\
                                        U_2(x) & U_1 (x) \\
\end{array}
\right)
$$ 
is an orthogonal matrix with entries $U_j \in L^\infty (D)$. Then, after the  complexification 
$$
u (z) = V_1 (z)+ \sqrt{-1} V_2 (z), \, z = x_1 +  \sqrt{-1} x_2
$$
system (\ref{eq.Lame.1}) with real-valued coefficients reduces to the following 
equation with complex-valued coefficients   
$$
A (x,\partial,\lambda)  u =   
4 \vartheta \overline \partial ^* \overline \partial \,  u + \lambda^2  a_0^{(2)} (x) u
$$
where $\overline \partial = 1/2 (\frac{\partial }{ \partial x_1} + \sqrt{-1} 
\frac{\partial }{ \partial  x_2})$ is the Cauchy-Riemann operator, 
$\overline \partial^*  = - 1/2 (\frac{\partial }{ \partial x_1} - \sqrt{-1} 
\frac{\partial }{ \partial  x_2})$ its formal adjoint and 
$$
a_0^{(2)} (x) = \alpha (x) \left( U_1 (x) + \sqrt{-1} U_2 (x) \right).
$$
Then, with a proper operator $\Psi: H^\rho (\partial D) \to L^2 (\partial D)$, the Robin 
type operator $B$ has the form  
$$ 
B =  2 \vartheta \, (\nu_1 - \sqrt{-1} \nu_2) \bar \partial + \Psi^* \Psi.
$$
where $(\nu_1, \nu_2)$ is the unit normal vector field to $\partial D$. 
The boundary operators 
$$ 
\frac{\partial}{\partial \nu}  = \nu_1 \partial_1 + \nu_2 \partial_2 , \, 
\bar \partial _\nu = 
(\nu_1 - \sqrt{-1} \nu_2) \bar{\partial} = \frac{1}{2} \left( \frac{\partial}{\partial \nu} + 
\sqrt{-1} (\nu_1 \partial_2 - \nu_2 \partial _1) \right)
$$ 
are known as the normal derivative and the complex normal 
derivative with respect to $\partial D$ respectively.  
Thus, we obtain a mixed problem of the type considered above:
\begin{equation} 
\label{eq.Helmholtz.0}
\left\{
\begin{array}{rclcl}
   \left( - \vartheta \Delta_2 + \lambda^2 a_0^{(2)} \right) u (z)
 & =
 & f 
 & \mbox{in}
 & D,
\\
   \left( 2 \vartheta  \bar{\partial}_\nu + \Psi^* \Psi \right) u  (z)
 & = & 0 & \mbox{at} & {\partial D}.
\end{array}
\right.
\end{equation}

Note that the usual  boundary conditions for the Navier-Stokes equations or the Lam\'e type 
operator  are formulated by using the  boundary stress tensor $\sigma$. In our particular case
the tensor have the following components:
\begin{equation} \label{eq.stress.tensor}
\sigma _{i,j}   =    \vartheta  \, \left( \delta_ {i,j} \frac{\partial }{\partial \nu}  +
  \nu_j \frac{\partial }{\partial x_i} -   \nu_i \frac{\partial }{\partial x_j} \right)
, \ 1\leq i,j \leq 2.
\end{equation}
Hence, with the tangential operator $\partial_{\tau_0} = \left( (\nu(x) \rm{div}_2 )^T -  \nu (x)
\rm{div}_2 \right)$, we have
\begin{equation} \label{eq.stress.var}
 \sigma  = \vartheta \, \left( \frac{\partial }{\partial \nu} \Id_2  + 
\partial_{\tau} \right) =\vartheta \, (\tilde \sigma    + 2  \partial_{\tau_0}) .
\end{equation}
where the boundary tensor $\tilde \sigma  $ corresponds to the  boundary operator 
$2\bar \partial_\nu$ after the decomplexification of the mixed problem (\ref{eq.Helmholtz.0}), 
i.e.  in the matrix form (\ref{eq.Helmholtz.0}) reads as 
\begin{equation*} 
\left\{
\begin{array}{rclcl}
   \left( - \vartheta \Delta_2 \Id_2 + \lambda^2 A_0^{(2)} \right) V (x)
 & =
 & F 
 & \mbox{in}
 & D,
\\
   \left(    (\sigma - 2 \vartheta \partial_{\tau_0})+ \Psi^* \Psi  \Id_2 \right) V (x)
 & = & 0 & \mbox{at} & {\partial D}.
\end{array}
\right.
\end{equation*}
In Elasticity Theory, the boundary tensor $ \tilde \sigma = 
\vartheta^{-1}\sigma - 2  \partial_{\tau_0}$  was discovered in \cite{Ca59}.

We continue with the mixed problem (\ref{eq.Helmholtz.0}). 
The corresponding scalar product of the space $H^{+} (D)$ related to the mixed 
problem has the form
$$
(u,v)_+ = 4 \vartheta \, (\bar \partial u, \bar \partial v)_{L^2 (D)} + 
(\Psi u, \Psi v)_{L^2 (\partial D)}.
$$
Then, Theorem \ref{t.emb.half} grants the embedding of the space
$H^{+} (D)$ into the  Sobolev-Slobodetskii space $H^{s} (D)$. 
However, for $0<\rho<1/2$ each holomorphic function $u \in H^{\rho + 1/2} (D) $ 
belongs to $H^+ (D)$ but there is no reason for it to belong to $H^1 (D)$, 
i.e. the embedding is sharp. For $\rho=0$ the embedding described in  
Theorem \ref{t.emb.half} is sharp, too  but the arguments a more subtle  
(see \cite{PolkShla13} or \cite{ShlTark12}).

In some cases we can obtain reasonable formulas for solutions to the problem.
Let  $D$ be the unit circle $\mathbb B$ around the origin in ${\mathbb C}$ and  
$S=\emptyset$. We pass to polar coordinates  $z = r \, e^{\sqrt{-1} \phi}$
in ${\mathbb R}^{2}$, where    $ r = |x|$ and  $\phi \in [0,2\pi]$.
The Laplace operator $\Delta_2$ takes the form
\begin{equation}
\label{eq.Laplace.spher}
   \Delta_2 
 = \frac{1}{r^2}
   \left( \left(r \frac{\partial }{\partial r} \right)^2  + \frac{\partial^2 }{\partial \phi^2}  
	\right).
\end{equation}
Furthermore, since $\partial D =\partial \mathbb{B}$, we get
$$ 
  \frac{\partial}{\partial \nu}
 = 
    r \partial_r, \, \, 
   \overline{\partial}_\nu =       \bar{z} \, \bar \partial 
 = \frac{1}{2}   \left( r \partial_r + \sqrt{-1}\, \partial_\phi \right). 
$$
As $\partial \mathbb B$ is smooth we may use powers of the Laplace-Beltrami operator 
on  $\partial \mathbb B$  as $\Psi^*\Psi$.  For simplicity, we set
$$
\vartheta=1, \, 
   \Psi^\ast \Psi\, 
 = 2\left(1 - \frac{\partial^2}{\partial ^2 \phi}\right)^{\rho/2} \quad 
a_0 ^{(2)} (z) = |z| ^{2d}, \, d \geq 0.
$$
If $0<d \leq 1/2$ then $a_0^{(2)}  \in C^{0,2d} (\overline D)$.

To solve the homogeneous equation
   $(- \Delta_2 +   \lambda^2 |z|^{2d} ) u = 0$
we apply the Fourier method of separation of variables.
Writing
   $u (r,\phi) = g (r) h (\phi)$
we get two separate equations for $g$ and
                                  $h$,
namely
$$
   \left( - (r \partial_ r)^2  + \lambda^2 r^{2(d+1)} \right) g  = 
   c\, g , \quad -\frac{\partial^2 h}{\partial \phi^2}   = 
   c\, h,
$$
where $c$ is an arbitrary constant.
The second equation possesses non-zero solutions if and only if
   $c = k^2$  and $h_k = e^{\sqrt{-1} k\phi}$.
In particular, 
\begin{equation}
\label{eq.Bessel.SL.0}
   \Psi^\ast \Psi \,  h_{k}
 = 2 \left(1 - \frac{\partial^2}{\partial ^2 \phi}\right)^{\rho/2}\, h_k 
 = 2 (1 + k^2)^{\rho/2} h_k, \,  \sqrt{-1}\partial_\phi h_k =-  k h_k, \, k \in {\mathbb Z}.
\end{equation}
%Clearly, we can identify $r^{|k|} h_k (\phi) = z^k$ for $k \in {\mathbb N}\cup \{0\} $ and 
%$r^{|k|} h_k (\phi) = \overline z^k$ for $k \in {\mathbb Z}_-$. Then 
%\begin{equation}
%\label{eq.Bessel.SL.0}
%\begin{array}{rcl}
%   \bar{\partial}_\nu (r^{|k|} h_k (\phi) ) =0, & \sqrt{-1}\partial_\phi h_k = -  k h_k, & 
%	k \geq 0,   \\ \bar{\partial}_\nu (r^{|k|} h_k (\phi) ) =k r^{|k|} h_k (\phi) , 
%	& \sqrt{-1}\partial_\phi h_k =-  k h_k, & k<0.
%\end{array}
%\end{equation}

Consider the Sturm-Liouville problem for the ordinary differential equation with respect to the
variable $r$ in the interval $(0,1)$,
\begin{equation}
\label{eq.Bessel}
\left( r \partial_r^2  + \partial _r  - k^2r^{-1}  + \mu^2 r^{2d+1}
   \right)   g = 0 \mbox{ in }  (0,1)
\end{equation}
\begin{equation}
\label{eq.Bessel.origin} 
g  \mbox{ is bounded at } 0,
\end{equation}
\begin{equation}
\textcolor{red}{
\label{eq.Bessel.unit} 
 \left(  r \partial_r  - k  +  (1 \! + \! k^2 )^{\rho/2}
   \right)    g  = 0 \mbox{ at } r=1}
\end{equation}
see \cite[Suppl.~II, Introduction and P.~1, \S~2]{TikhSamaX}.
Actually, as we have seen above $\mu$ are non-negative  
real numbers  (with $\mu^2 = - \lambda^2$) and then (\ref{eq.Bessel}) is a particular case
of the Bessel equation. Its (real-valued) solution $g (r)$ is a Bessel function
defined on $(0,+\infty)$, and the space of all solutions is two-dimensional.
For example, if $\lambda^2 = 0$ and $d=0$ then
   $g (r) = \alpha {r}^{k} + \beta {r}^{-k}$
with arbitrary constants $\alpha$ and
                         $\beta$
is a general solution to (\ref{eq.Bessel}).
In the general case the space of solutions to (\ref{eq.Bessel}) contains a one-dimensional
subspace of functions bounded at the point $r = 0$, cf. \cite{TikhSamaX}:
$$
g_k (r,\mu) =  {\mathcal J}_{\frac{|k|}{d+1}} \left(\frac{\mu r^{d+1}}{d+1}\right).
$$
where ${\mathcal J}_p (t)$  
%$$
%{\mathcal J}_p (t) = \left(\frac{t}{2}\right)^p \, 
%\sum_{j=0}^\infty \frac{ (-1)^j \, t^{2j}}{j! \, 2^{2j+p} \, \Gamma (p+j+1)}.  
%$$ 
are Bessel functions (see, for instance, \cite{Bo}). 
%Then, after direct calculations, we obtain 
%$$
%g_k (r,\mu) = 
%\left(\frac{\mu }{2(d+1)} \right)^{\frac{|k|}{d+1}} r^{|k|} 
%\sum_{j=0}^\infty \frac{ (-1)^j }{j! \, 2^{2j+|k|/(d+1)} 
%\, \Gamma (\frac{|k|}{d+1}+j+1)} 
%\left(\frac{\mu \, r^{d+1}}{d+1}\right)^{2j}.  
%$$
As usual, for each $k \in \mathbb Z$ the proper system of eigenvalues 
$\{ \mu^{(\nu)}_k \}_{
\nu \in \mathbb N}$ can be found as solutions to the transcendental equation 
$$ 
\frac{\mu }{d+1}  {\mathcal J}^\prime_{\frac{|k|}{d+1}} \left(\frac{\mu }{d+1}  \right)
+ \left((1 \! + \! k^2 )^{\rho/2} - k   \right)   
{\mathcal J}_{\frac{|k|}{d+1}} \left(\frac{\mu }{d+1}  \right) = 0 
$$
induced by (\ref{eq.Bessel.unit}) with $g_k (\cdot,\mu)$ instead of $g$.

For any $k \in \mathbb Z$, fix a non-trivial solution $g^{(\nu)}_{k} 
(r)$ of problem (\ref{eq.Bessel}) corresponding to an eigenvalue $\mu^{(\nu)}_{k}$.
This system is an orthogonal basis in the weighted Lebesgue space  $L^2_d (0,1)$ 
with the scalar product 
$$
h_{d} (g,f) = \int_0^1 r^{2d+1} g (r) f (r) \ dr, 
$$
see \cite[Suppl.~II, Introduction and P.~1, \S~2]{TikhSamaX}.
 Then the function
$$
   u^{(\nu)}_{k} (z)= g^{(\nu)}_{k} (r) e^{\sqrt{-1}k \phi} 
$$
satisfies
\begin{equation}
\label{eq.Helmholtz.1}
\left\{
\begin{array}{rclcl}
   \left( - \Delta_2 + (\lambda^{(\nu)}_{k})^2 |z|^{2d} \right) u^{(\nu)}_{k} (z)
 & =
 & 0
 & \mbox{in}
 & \C,
\\
   \left( \bar{\partial}_\nu +(1 - \frac{\partial^2}{\partial ^2 \phi})^{\rho/2} \right) u^{(k)}_{\nu} (z)
 & = & 0 & \mbox{at} & \partial {\mathbb B}.
\end{array}
\right.
\end{equation}
where $(\lambda_k^{(\nu)})^2 = - (\mu_k^{(\nu)})^2 $
Indeed, by %(\ref{eq.spher.harm}),
           (\ref{eq.Laplace.spher}),
           (\ref{eq.Bessel})
and the discussion above we conclude that this equality holds in
   ${\mathbb C}\setminus \{ 0 \}$.
We now use the fact that $u^{(k)}_{\nu}$ is bounded at the origin to see that the differential
equation of (\ref{eq.Helmholtz.1}) holds in all of $\mathbb C$. On the other hand, the 
boundary condition (\ref{eq.Helmholtz.1}) follows from (\ref{eq.Bessel.SL.0}) 
immediately, as already mentioned. 
%After direct calculations we obtain 
%$$
%  u^{(\nu)}_{k} (z)   = \left\{ 
%	\begin{array}{ll}
%\left(\frac{\sqrt{-1} \lambda^{(\nu)}_k }{2(d+1)} \right)^{\frac{k}{d+1}} z^{k} 
%\sum_{j=0}^\infty \frac{ (\lambda^{(\nu)}_k)^{2j} \, (\overline z \, z)^{j(d+1)} }
%{2! \, 2^{2j+k/(d+1)} \, \Gamma (j+1+\frac{k}{d+1}) (d+1)^{2j}}, k \in {\mathbb Z}_+
% \\ \left(\frac{\sqrt{-1} \lambda^{(\nu)}_k }{2(d+1)} \right)^{\frac{-k}{d+1}} 
%\overline z^{k} \sum_{j=0}^\infty \frac{ (\lambda^{(\nu)}_k)^{2j} \, (\overline z 
%\, z)^{j(d+1)} }{j! \, 2^{2j-k/(d+1)} \, \Gamma (j+1-\frac{k}{d+1}) (d+1)^{2j}}, 
%k \in {\mathbb Z}_- .
%\end{array}
%\right.  
%$$
Now, by the construction,  the system
$
   \{ u^{(k)}_{\nu} \}_{k \in {\mathbb Z}, \nu \in {\mathbb N}} 
$
consists of eigenfunctions of the family $L (\lambda) = L_0 + \lambda^2 C$  
in the case of the unit ball $\mathbb{B}$ around the origin in $\mathbb C$. Obviously, 
it coincides with the system of all eigenvectors constructed in Theorem 
\ref{t.PE.self.complete} if it is complete in the space $L^2 _h (\mathbb{B})$ with the 
scalar product 
$$
h (u,v) = \int_D |z|^{2d} u(z) \overline v (z) dx .
$$
But  $\{h_k\}_{k \in {\mathbb Z}}$ is an orthogonal basis 
in $L^2 (\partial {\mathbb B})$ and $\{ g^{(\nu)}_{k} \}_{k \in {\mathbb Z}_+, \nu \in 
{\mathbb N}}$  is an orthogonal basis in  $L^2_d (0,1)$ and hence  Fubini Theorem implies 
that the system is orthogonal basis in the space $L^2_h ({\mathbb B})$.

\bigskip

{\sc Acknowledgements\,}
The work was supported by the grant of the Russian Federation Government
 for scientific research under the supervision of leading scientist
 at the Siberian federal university, contract N. 14.Y26.31.0006,
 and by RFBR grant 14-01-00544.


\begin{thebibliography}{25}

\bibitem{Agmo62}
{\sc Agmon, S.},
  {\em On the eigenfunctions and on the eigenvalues of general elliptic
          boundary value problems},
  Comm. Pure Appl. Math. \textbf{15} (1962), 119--147.

\bibitem{ADN59}
{\sc Agmon, S., Douglis, A.},  {\sc Nirenberg, L.},
{\em Estimates near the boundary for solutions of elliptic partial differential equations
          satisfying general boundary conditions. P. 1},
  Comm. Pure Appl. Math. \textbf{12} (1959), 623--727.

\bibitem{AgraVi64} %[AV64]
{\sc Agranovich, M.~S.},  {\sc Vishik, M. I.}, 	
  \textit{Elliptic problems with a parameter and parabolic problems of general type},
   Uspekhi Mat. Nauk \textbf{19} (1964), No.~117 (3), 53?161.

\bibitem{Agra11a}
{\sc Agranovich, M.~S.}, 
  {\it Spectral Problems in Lipschitz Domains},
  In: Modern Mathematics, Fundamental Trends \textbf{39} (2011), 11--35.

\bibitem{Bo} {\sc  Bowman, F.,}  
\textit{Introduction to Bessel Functions}, New York: Dover, 1958.

\bibitem{Brow53}
{\sc Browder, F.~E.},
 \textit{On the eigenfunctions and eigenvalues of the general elliptic
         differential operator},
 Proc. Nat. Acad. Sci. USA \textbf{39} (1953), 433--439.

\bibitem{Ca59} 
{\sc Campanato~S.},  \textit{Sui problemi al contorno per sistemi
di equazioni differenziale lineari del tipo dell'elasticit\'a}, Ann. della Scuola
Norm. Superiore, Cl. di Sci, Ser. III, \textbf{13:2}, pp.\,223--258 (1959).


\bibitem{DMV01}
{\sc Denk~R.,  Mennicken~R., Volevich~L.}, {\em 
On elliptic operator pencils with general boundary conditions}, 
Integral Equations and Operator Theory 
 {\bf 39:1} (2001), 15-40. 

\bibitem{DV02}
{\sc Denk~R., Volevich~L.}, {\em Parameter-elliptic boundary value problems connected with the Newton polygon}, Differential Integral Equations
V. {\bf 15:3} (2002), 289-326.



\bibitem{DunfSchw63}
{\sc Dunford, N.},  {\sc Schwartz, J.~T.},
  \textit{Linear Operators, Vol. II, Selfadjoint Operators in Hilbert Space},
  Intersci. Publ., New York, 1963.

%\bibitem{EgorKondSchu01}
%{\sc Egorov,~Yu., Kondratiev,~V.}, and {\sc Schulze, B.~W.},
%  \textit{Completeness of eigenfunctions of an elliptic operator on a manifold
%          with conical points},
%  Russ. J. Math. Phys. \textbf{8} (2001), No.~3, 267--274.

\bibitem{Fi72}
{\sc Fichera,~G.},   
\textit{Existence Theorems in Elasticity}, in: Festk\"orpermechanik/Mechanics of
Solids, edited by S.~Fl\"ugge, C.A.~Truesdell,
Handbuch der Physik (Berlin--Heidelberg--New York, Springer--Verlag, 1972), pp.\,347--389.

\bibitem{GokhKrei69}
{\sc Gokhberg, I.~Ts.}, {\sc Krein, M.~G.},
  \textit{Introduction to the Theory of Linear Nonselfadjoint Operators in Hilbert
          Spaces},
  AMS, Providence, R.I., 1969.

\bibitem%[GS71]
{GokhSiga1}
{\sc Gokhberg,~I. Ts.}, and {\sc Sigal,~E. I.},
  \textit{An operator generalisation of the logarithmic residue theorem and
          the theorem of Rouch\'{e}},
  Math. USSR Sbornik {\bf 13} (1971), 603--625.

\bibitem{Gris85}
{\sc Grisvard,~P.},
{\em Elliptic Problems in Non-Smooth Domains},
Pitman, Boston, 1985.

\bibitem{Ho85}
{\sc H\"ormander~L.,} {\em The Analysis of Linear Partial Differential Operators}, 
Vols. 3--4, Springer-Verlag, New York, 1985.

\bibitem{Keld51}
{\sc Keldysh, M.~V.},
  \textit{On the characteristic values and characteristic functions of certain
          classes of non-selfadjoint equations},
 Dokl. AN SSSR \textbf{77} (1951), 11--14.

\bibitem{LadyUral73}
{\sc Ladyzhenskaya, O.~A.}, {\sc Uraltseva, N.~N.},
  \textit{Linear and Quasilinear Equations of Elliptic Type},
  Nauka, Moscow, 1973.

\bibitem{LaLi59} {\sc Landau,~L.~D, Lifshitz,~E.~M.}
  \textit{ Fluid Mechanics, Volume 6 of A Course of Theoretical Physics}, 
Pergamon Press, London--New York--Paris, 1959.


\bibitem{LiMa72}
{\sc  Lions, J. L.}, {\sc Magenes, E.},
\textit{ Non-Homogeneous Boundary Value Problems und Applications. Vol. 1},
Springer-Verlag, Berlin et al., 1972.

\bibitem{Mar86}
{\sc Markus~A.S.} {\em  
Introduction to the Spectral Theory of Polynomial Operator Pencils}, 
Translations of Mathematical Monographs,  AMS, V. 71, 1988, Providence, Rhode Island.


\bibitem{Pal96} 
{\sc Paltsev, B.~V.},
  \textit{Mixed problems with non-homogeneous boundary conditions in Lipschitz domains for
          second-order elliptic equations with a parameter},
  Mat. Sb., \textbf{187} (1996), No.~4, 59-116.

\bibitem{PolkShla13}
{\sc Polkovnikov, A.},  {\sc Shlapunov, A.},
  \textit{On the spectral properties of a non-coercive mixed problem associated with
          $\overline \partial$-operator},
  J. Siberian Fed. Uni. \textbf{6} (2013), No.~2.

\bibitem{ShlTark12}
{\sc Shlapunov,~A.},  {\sc Tarkhanov, N.},
  \textit{On completeness of root functions of Sturm-Liouville problems with discontinuous
         boundary  operators}, J. of Differential Equations, {\bf 10} (2013), 3305--3337.

%\bibitem{ShlTark13}
%{\sc Shlapunov~A., Tarkhanov~N.}
%\textit{Sturm-Liouville  problems  in domains with non-smooth edges}. 
%Preprints des Instituts f\"ur Mathematik der Universit\"at Potsdam
%2 (2013) 13, 1-110, ISSN (online) 2193-6943.

\bibitem{Tark06}
{\sc Tarkhanov,~N.},
 \textit{On the root functions of general elliptic boundary value problems},
 Compl. Anal. Oper. Theory \textbf{1} (2006), 115--141.

\bibitem{TikhSamaX}
{\sc Tikhonov,  A.~N.}, {\sc Samarskii, A.~A.},
 \textit{Equations of Mathematical Physics},
 Nauka, Moscow, 1972.

\bibitem{Yak90} {\sc Yakubov, S. Ya.} {\it 
Multiple completeness for systems of operator bundles and elliptic boundary value  
problem}, Matem. Sb., 181:1 (1990),  95--113.

\bibitem{Zare10}
{\sc Zaremba, S.},
\textit{Sur un probl\`{e}me mixte relatif \`{a} l'\'{e}quation de Laplace},
Bull. Acad. Sci. Cracovie (1910), 314--344.

\end{thebibliography}
\end{document}